\documentclass{article}
\usepackage{amssymb,amsmath,amsfonts,amsthm,latexsym}
 \usepackage[all,cmtip]{xy}
\usepackage[dvips]{graphicx}

\usepackage[latin1]{inputenc}
\usepackage{amscd}
\usepackage{enumerate}
\usepackage{hyperref}

\usepackage[usenames,dvipsnames]{pstricks}
\usepackage{epsfig}
\usepackage{pst-grad} 
\usepackage{pst-plot} 
\usepackage{pstricks-add}
\usepackage{pst-solides3d}

\usepackage{caption}
\usepackage{subcaption}

\newtheorem{theo}{Theorem}[section]

\newtheorem{coro}[theo]{Corollary}
\newtheorem{prop}[theo]{Proposition}
\newtheorem{proposition}[theo]{Proposition}
\newtheorem{lemm}[theo]{Lemma}
\newtheorem{lemma}[theo]{Lemma}

\theoremstyle{definition}
\newtheorem{example}[theo]{Example}
\newtheorem{remark}[theo]{Remark}
\newtheorem{rema}[theo]{Remark}
\newtheorem{defi}[theo]{Definition}
\newtheorem{definition}[theo]{Definition}
\newtheorem{question}[theo]{Question}

\numberwithin{equation}{section}

 \def\RR{{\mathbb R}}  \def\TT{{\mathbb T}}
   
 \def\ZZ{{\mathbb Z}}

\newcommand{\bH}{\mathbb{H}}
\newcommand{\bD}{\mathbb{D}}
\newcommand{\R}{\mathbb{R}}

\newcommand{\cF}{\mathcal{F}}
\newcommand{\cU}{\mathcal{U}}

\newcommand{\eps}{\varepsilon}

\newcommand{\cV}{\mathcal{V}}

\def\cC{{\cal C}}    \def\cU{{\cal U}}
    \def\cV{{\cal V}}
\def\cE{{\cal E}}    
\def\cF{{\cal F}}

\newcommand{\wt}[1]{\widetilde{#1}}
\newcommand{\ie}{{\it i.e.,} }

\newcommand\flot{ \varphi^{t} }
\newcommand\hflot{ \tilde{ \varphi}^t }

\newcommand\orb{ \mathcal O }
\newcommand\leafs{ \mathcal L ^{s} }
\newcommand\leafu{ \mathcal L ^{u} }
\newcommand\e{\varepsilon}

\newcommand\fs{\mathcal F^{s} }
\newcommand\hfs{\widetilde{\mathcal F}^{s} }
\newcommand\fu{\mathcal F^{u} }
\newcommand\hfu{\widetilde{\mathcal F}^{u} }

\renewcommand\phi{\varphi}

\title{Anomalous Anosov flows revisited}
\author{Thomas Barthelm\'e, Christian Bonatti, \\ Andrey Gogolev and Federico Rodriguez Hertz \thanks{Part of this article was done during a visit of C.B.~to Pennnsylvania State University, supported by the Centre for Dynamical Systems and Geometry - Penn State and NSF DMS-1500947. A.G.~was partially supported by Simons grant 427063. F.R.H.~was partially supported by NSF DMS-1500947.} }

\begin{document}

\maketitle

\begin{abstract}
This paper is devoted to higher dimensional Anosov flows and consists of two parts. In the first part, we investigate fiberwise Anosov flows on affine torus bundles which fiber over 3-dimensional Anosov flows. We provide a dichotomy result for such flows --- they are either suspensions of Anosov diffeomorphisms or the stable and unstable distributions have equal dimensions.

In the second part, we give a new surgery type construction of Anosov flows, which yields non-transitive Anosov flows in all odd dimensions.

{ \medskip \noindent \textbf{Keywords:} Anosov flow, fiberwise Anosov flow, geodesic flow, hyperbolic manifold, DA flow, surgery.

  }
\end{abstract}


\section{Introduction}

Anosov dynamical systems have been fascinating many mathematicians by their uniform and regular structure leading to chaotic dynamics. All known examples of Anosov diffeomorphism are transitive, that is, they have an orbit which is dense in the ambient manifold. 
One early surprise in the field was the 1979 example of Franks and Williams of a non-transitive Anosov flow.
\begin{theo}[\cite{FW}]
There exists a 3-dimensional closed manifold $M$ which admits a non-transitive Anosov flow.
\end{theo}

Further, Theorem~1.1b) of~\cite{FW} asserts that building upon the 3-dimensional example one can obtain non-transitive Anosov flows on manifolds of higher dimension $\ge 5$ with dimensions of stable and unstable distributions --- $s$ and $u$ --- taking arbitrary values $\ge 2$. (Recall that codimension one Anosov flow on manifold of dimension $\ge 4$ are transitive by work of Verjovsky~\cite{Ve}.)

In this article, we revisit Franks-Williams examples and point out a deficiency in their higher dimensional construction. First we put Franks-Williams suggestion into a more general context of fiberwise Anosov flows on affine torus bundles (see the next section for a precise definition) and prove the following result about fiberwise Anosov flows.
\begin{theo}
\label{thm1}
Let $M$ be a closed $3$-dimensional manifold and let $\varphi^t\colon M\to M$ be an Anosov flow. Let $\TT^d\to E\to M$ be an affine torus bundle and let $\Phi^t\colon  E\to E$ be a fiberwise Anosov flow which fibers over $\varphi^t$. Then $\Phi^t$ is  Anosov (as a flow on the total space $E$) and one of the following assertions must hold
\begin{enumerate}
\item[1.] flow $\varphi^t$ is topologically orbit equivalent to a suspension flow of an Anosov automorphism of the 2-torus $\TT^2$, and $\Phi^t$ is also a suspension of an Anosov diffeomorphism;
\item[2.] the dimensions of stable and unstable distributions of the Anosov flow $\Phi^t$ are equal and $\ge 3$:
$$
s=u\ge 3.
$$
\end{enumerate}
\end{theo} 
We will explain that both possibilities above indeed occur. Note that because suspension flows are transitive the above theorem precludes the existence of non-transitive higher dimensional Anosov flows which fiber over a non-transitive 3-dimensional Anosov flow when $s\neq u$.

Theorem \ref{thm1} is a corollary of the following result.
\begin{theo}
\label{thm_fiberwise_anosov_over_n-dimensional_anosov}
Let $M$ be a closed manifold and let $\varphi^t\colon M\to M$ be an Anosov flow. Let $\TT^d\to E\to M$ be an affine torus bundle and let $\Phi^t\colon  E\to E$ be a fiberwise Anosov flow which fibers over $\varphi^t$. 
Assume that there exists two periodic orbits $a$ and $b$ (possibly $a=b$) of $\varphi^t$ such that $a$ is freely homotopic to $-b$, \ie to the orbit $b$ with reversed orientation, then $d$ has to be even, $d\ge 4$, and the fiberwise stable and unstable distributions of $\Phi^t$ have equal dimensions. 
\end{theo}

\begin{rema}
Just as in Theorem~\ref{thm1} the flow $\Phi^t$ is Anosov as a flow on the total space $E$ (see Proposition~\ref{prop_anosov}). The fiberwise stable and unstable distributions of $\Phi^t$ are the restrictions of the stable and unstable distributions to the tangent space of the fibers. Hence, if $\phi^t$ has matching stable and unstable dimensions (so, in particular, if $\phi^t$ is a 3-dimensional flow, in which case the stable and unstable distributions are both one-dimensional) then the stable and unstable distributions of $\Phi^t$ also have equal dimensions.
\end{rema}

Work of Barbot and Fenley (see Theorem~\ref{t.free} below) implies that the only Anosov flows on $3$-manifolds that do not admit pairs of periodic orbits which are freely homotopic to each others' inverse are orbit equivalent to suspensions of Anosov automorphisms. This allows us to deduce Theorem~\ref{thm1} from Theorem~\ref{thm_fiberwise_anosov_over_n-dimensional_anosov}.

Another application of Theorem \ref{thm_fiberwise_anosov_over_n-dimensional_anosov} concerns \emph{algebraic} Anosov flows. An Anosov flow $\Phi^t$ is called algebraic if it can be seen as a $\R$-action on a homogeneous manifold $H \backslash G / \Gamma$, where $G$ is a connected Lie group, $H$ a compact subgroup of $G$ and $\Gamma \subset G$ a torsion-free uniform lattice. Tomter~\cite{T,T75} (see also~\cite{BM13}) gave a certain classification of algebraic Anosov flows and supplied examples. We can use this classification of Tomter and Theorem~\ref{thm_fiberwise_anosov_over_n-dimensional_anosov} to deduce the following.
%
\begin{coro} \label{cor_algebraic_flows}
 Let $\Phi^t$ be an algebraic Anosov flow. Then either the dimensions of the stable and unstable distributions of $\Phi^t$ are equal, or $\Phi^t$ is (up to passing to a finite cover) a suspension of an Anosov automorphism.
\end{coro}

In the second part of this article, we proceed to recover some non-transitive examples in the case when $s=u$ by replanting the Franks--Williams idea into the setting of geodesic flows on hyperbolic manifolds.
\begin{theo}
\label{thm2}
For any $n\ge 1$ there exists a $(2n+1)$-dimensional manifold $M$ which supports a non-transitive Anosov flow $\psi^t\colon M\to M$ with $s=u=n$.
\end{theo}

\begin{remark}
It seems to be clear that by following our surgery idea for the proof of the above theorem one can obtain further examples of Anosov flows (transitive and non-transitive) in odd dimensions by pasting together several ``hyperbolic pieces." Indeed, one might even hope to fully develop the machinery of B\'eguin--Bonatti--Yu~\cite{BBY} in any odd dimension. However, it is clear that the proof of the Anosov property becomes much more subtle and we do not pursue this direction here. However we will give, without providing all the details, one relatively easy way of building a transitive example (Proposition~\ref{prop_transitive_examples})
\end{remark}

As we mentioned above, in dimension $3$, the only Anosov flows which do not admit periodic orbits freely homotopic to each others' inverse are the suspensions of Anosov diffeomorphisms. Tomter's classification of algebraic Anosov flows also implies that only suspensions do not satisfy that property. Finally, it is also easy to notice that the examples we build in Theorem~\ref{thm2} all admit a pair of periodic orbits which are freely homotopic to each others' inverse, and this would be the case for any example build with a similar surgery construction. Hence, the following question arises naturally.
 \begin{question}\label{Q_non_freely_hom}
  Does there exists an Anosov flow (on a manifold of dimension at least $4$), which is not a suspension of an Anosov diffeomorphism, and such that it does not admit a pair of periodic orbits which are freely homotopic to each others' inverse?
   \end{question}

 In order to obtain Theorem \ref{thm_fiberwise_anosov_over_n-dimensional_anosov}, we strongly rely on the fiberwise Anosov structure. However, it is not clear that it is necessary to have this structure in order to deduce the equality of the stable and unstable dimensions. Thus, we ask
 \begin{question}\label{Q_hom_and_dimensions}
  Let $\varphi^t$ be an Anosov flow on a manifold $M$. Suppose that $\varphi^t$ admits a pair of periodic orbits that are freely homotopic to the inverse of each other. Does this imply that the stable distribution and the unstable distribution have the same dimension?
 \end{question}

 In light of the situation with algebraic flows, one can even combine the above two questions, and wonder about a ``Generalized Verjovsky Conjecture." We pose it as a question.

\begin{question} \label{Q_generalized_Verjovsky}
Does there exist an Anosov flow which is not orbit equivalent to a suspension and which has different dimensions of stable and unstable distributions?
\end{question}
In particular, we wonder about the existence of \emph{non-transitive} examples. Indeed, if Questions \ref{Q_non_freely_hom} and \ref{Q_hom_and_dimensions} lead naturally to the general case of Question \ref{Q_generalized_Verjovsky}, there is another way of arriving to that question. Verjovsky \cite{Ve} proved that codimension one Anosov flow in dimension strictly greater than $3$ are transitive. So one can ask whether the essential feature for Verjovsky's result to hold is indeed the codimension one hypothesis or if the difference of stable and unstable dimensions is enough (notice that we are asking about the result, Verjovsky's proof do use codimension one in an essential way).


In Section~\ref{sec_2} we describe the setting of fiberwise Anosov flows, revisit the Franks-Williams construction and then prove Theorems~\ref{thm1} and \ref{thm_fiberwise_anosov_over_n-dimensional_anosov}.
In Section~\ref{sec_3}, we explain the construction which yields Theorem \ref{thm2}.

\section{Franks--Williams and fiberwise Anosov flows}

\label{sec_2}

\subsection{Definition of fiberwise Anosov flows and examples}
Let $M$ be a closed smooth manifold and let $\TT^d\to E\stackrel\pi\to M$ be an {\it affine torus bundle}, that is, a locally trivial fiber bundle with structure group $SL(d,\ZZ)\ltimes\TT^d$ acting on $\TT^d$ by affine transformations $(A,v)x=Ax+v$. Denote by $\TT^d_x$ the fiber over $x\in M$ and by $VE = \ker D\pi$ the {\it vertical subbundle} of the tangent bundle $TE$, which consists of vectors tangent to the fibers of $\pi$.

\begin{defi} 
\label{def_anosov}
Given an affine torus bundle $\TT^d\to E\stackrel\pi\to M$ and a flow $\varphi^t\colon M\to M$, a flow $\Phi^t\colon E\to E$ is called a {\it fiberwise Anosov flow} over $\varphi^t\colon M\to M$ if the following two conditions hold:
\begin{enumerate}
\item $\Phi^t$ fibers over $\varphi^t$, that is, the diagram 
$$
\xymatrix{
E\ar_\pi[d]\ar^{\Phi^t}[r] & E\ar_\pi[d]\\
M\ar^{\varphi^t}[r] & M
}
$$
commutes;
\item there exists a $D\Phi^t$-invariant splitting $VE=V^s\oplus V^u$ a constant $C>0$ and a constant $\lambda\in(0,1)$ such that for $t>0$
\begin{align*}
\|D\Phi^tv^s\|\le C\lambda^t\|v^s\|,\,\, v^s\in V^s \\
 \|D\Phi^{-t}v^u\|\le C\lambda^t\|v^u\|,\,\, v^s\in V^u
\end{align*}
\end{enumerate}
\end{defi}
(Cf. the definition of fiberwise Anosov flow in~\cite{FG}.) 

Among the fiberwise Anosov flows, we single out the following type
\begin{defi}
A flow $\Phi^t\colon E\to E$ is called a fiberwise \emph{affine} Anosov flow over $\varphi^t\colon M\to M$ if it is a fiberwise Anosov flow and, for each $x\in M$ and $t\in\RR$, the map $\Phi^t\colon\TT^d_x\to\TT^d_{\varphi^t(x)}$ is an affine diffeomorphism (note that the property of being affine is independent of the choice of trivializing charts at $x$ and $\varphi^t(x)$ because of our assumption on the structure group);
\end{defi}

\begin{rema}
It follows from the definition that the splitting $VE=V^s\oplus V^u$ is also affine.
\end{rema}
In this article, we will only be interested in the case when the base flow $\varphi^t$ is Anosov. By using the standard cone argument one can easily verify the following.
\begin{prop}
\label{prop_anosov}
Let $\Phi^t\colon E\to E$ be a fiberwise Anosov flow which fibers over an Anosov flow $\varphi^t\colon M\to M$. Then $\Phi^t$ is Anosov as a flow on the total space $E$. 
\end{prop}

The following partial converse to Proposition~\ref{prop_anosov} also holds. 
\begin{prop}
\label{prop_anosov2}
Let $\Phi^t\colon E\to E$ be a fiberwise (affine) flow which fibers over an Anosov flow $\varphi^t\colon M\to M$. Assume that $\Phi^t$ is Anosov (as a flow on the total space $E$) then $\Phi^t$ is a fiberwise (affine) Anosov flow, \ie also satisfies condition~2 of Definition~\ref{def_anosov}.
\end{prop}
\begin{proof} Using the fact that both $\Phi^t$ and $\varphi^t$ are Anosov one can easily verify that $D\pi\colon TE\to TM$ respects the Anosov splittings $TE=E^s_\Phi\oplus X_\Phi\oplus E^u_\Phi$ and $TM=E^s_\varphi\oplus X_\varphi\oplus E^u_\varphi$ and that the restrictions $D\pi\colon E^s_\Phi\to E^s_\varphi$, $D\pi\colon X_\Phi\to X_\varphi$ and $D\pi\colon E^u_\varphi\to E^u_\varphi$ are onto. We let
$$
V^s=\ker(D\pi \colon E^s_\Phi\to E^s_\varphi),\,\,\, V^u=\ker(D\pi \colon E^u_\Phi\to E^u_\varphi).
$$
Clearly $V^s$ and $V^u$ are subbundles of $VE$ which are transverse and satisfy the Anosov property. It remains to notice that by dimension count $V^s$ and $V^u$ span $VE$.
\end{proof}


\begin{example}
\label{ex_suspension}
Let $A\colon \TT^d\to\TT^d$ be a hyperbolic automorphism. Let
$$
M_A= \TT^d\times [0,1]/(x,1) \sim (Ax, 0)
$$
be its mapping torus and $\Phi^t_A\colon M_A\to M_A$ the suspension flow given by $\Phi^t_A\colon(x,s)\mapsto (x,s+t) \mod\ZZ$, where $\ZZ$ is the Deck covering group of the obvious covering $\TT^d\times\RR\to M_A$. Then $\Phi^t$ is a fiberwise affine Anosov flow which fibers over $\varphi^t\colon S^1\to S^1$ given by $\varphi^t\colon s\mapsto s+t \mod\ZZ$.
\end{example}

\begin{example}
Let $B\colon \TT^d\to\TT^d$ and $A\colon \TT^n\to\TT^n$ be two hyperbolic automorphisms. Let $\Phi^t_A\colon M_A\to M_A$ and $\Phi^t_{B\times A}\colon M_{B\times A}\to M_{B\times A}$ be the suspension flows of $A$ and $B\times A$, respectively. The map $\tilde\pi\colon \TT^d\times \TT^n\times\RR\to\TT^n\times\RR$ given by $(x,y,s)\mapsto (y,s)$ is a trivial affine $\TT^d$ bundle. This bundle map, obviously, commutes with the $t$-translation in the last coordinate and intertwines the $\ZZ$-action of the Deck group on $\TT^d\times\TT^n\times\RR$ with the $\ZZ$-action of  the Deck group on $\TT^n\times\RR$. Hence $\pi=\tilde\pi/\ZZ\colon M_{B\times A}\to M_A$ is an affine $\TT^d$ bundle and $\Phi^t_{B\times A}$ is a fiberwise affine Anosov flow over the suspension flow $\Phi^t_A$. When $n=2$ this provides examples for Case~1 of Theorem~\ref{thm1}.
\label{ex_suspension2}
\end{example}

\begin{example}
\label{ex_tomter}
Tomter~\cite{T} exhibited an algebraic Anosov flow $\Phi^t\colon G/\Gamma\to G/\Gamma$ of ``mixed type", where $G=PSL(2,\RR)\ltimes\RR^4$ and $\Gamma=\Gamma_1\ltimes \ZZ^4$ with $\Gamma_1<PSL(2,\RR)$ a cocompact lattice. If we view $G/\Gamma$ as an affine $\TT^4$ bundle over $PSL(2,\RR)/\Gamma_1$ then $\Phi^t$ can be viewed as a fiberwise affine Anosov flow over the geodesic flow $\varphi^t\colon PSL(2,\RR)/\Gamma_1\to PSL(2,\RR)/\Gamma_1$. For this example the dimensions of stable and unstable distributions are equal to 3, \ie $s=u=3$. In particular this implies  that the conclusion of Case~2 on Theorem~\ref{thm1} is optimal.

More examples of this type can be build. A general method is described in~\cite[Section 3.4.2]{BM13}.
\end{example}

Another open question regarding fiberwise Anosov flows is the following.
\begin{question}
Does there exist a non-transitive Anosov flow $\varphi^t\colon M\to M$ on a manifold $M$ and a fiberwise affine Anosov flow $\Phi^t\colon E\to E$ which fibers over $\varphi^t$ with $s=u\ge 3$?
\end{question}

Finally, we notice that a fiberwise Anosov flow over a suspension of an Anosov diffeomorphism is itself a suspension.
\begin{prop} \label{prop_suspensions}
 Let $\Phi^t\colon E\to E$ be a fiberwise Anosov flow which fibers over an Anosov flow $\varphi^t\colon M\to M$. Suppose $\varphi^t$ is a suspension of an Anosov diffeomorphism $f\colon N \to N$. Then $\Phi^t$ is a suspension of an Anosov diffeomorphism $\hat f \colon \hat{N} \to \hat{N}$, where $\hat{N}$ is a manifold that fibers over $N$ with fibers $\TT^d$.\end{prop}

\begin{proof}
We view $N\subset M$ as a section for $\varphi^t$. Let $\hat N=\pi^{-1}(N)$. Then, clearly, $\hat N$ is a section for $\Phi^t$ and the return map is an Anosov diffeomorphism.
\end{proof}

\begin{remark}
When $\phi^t$ is a suspension of an Anosov nilmanifold automorphism (e.g., when $M$ is 3-dimensional), with more work, one can show that $\hat{N}$ is finitely covered by a nilmanifold. Then $\Phi^t$ is orbit equivalent to the suspension of an affine Anosov diffeomorphism on an infranilmanifold.
\end{remark}



\subsection{The Franks-Williams construction}

In this section we briefly recall the beautiful 3-dimensional Franks-Williams construction omitting all technical details. Then we explain that the higher dimensional construction suggested in~\cite{FW}, if they exists, belong to the class of fiberwise affine flows. Further we point out the problematic step in the higher dimensional construction and discuss relations to our Theorem~\ref{thm1}.

Let $A\colon\TT^2\to\TT^2$ be a hyperbolic automorphism. Deform $A$ in a neighborhood of a fixed point $p$ to obtain a $DA$ diffeomorphism $f\colon\TT^2\to\TT^2$ such that the non-wandering set of $f$ consists of the repelling fixed point $p$ and a 1-dimensional hyperbolic attractor whose stable foliation $W^s_f$ is contained in the stable foliation of $A$ and, hence, is linear.  Consider the suspension flow $\varphi^t\colon M_f\to M_f$  (similarly to Example~\ref{ex_suspension}). The key idea of Franks and Williams is to remove a small tubular neighborhood of the repelling periodic orbit (corresponding to $p$) and then to equip certain twisted double of this manifold with a non-transitive Anosov flow. Namely, after removing the tubular neighborhood we obtain a manifold $\bar M_f$ with torus boundary which is equipped with the flow $\varphi^t$ whose orbits are transverse to the boundary and flow ``inwards." Also consider another copy of $\bar M_f$, say $\bar{\bar M}_f$, which is equipped with $\varphi^{-t}$ whose orbits are transverse to the boundary and flow ``outwards."  Then we paste together $(\bar M_f,\varphi^t)$ and $(\bar{\bar M}_f, \varphi^{-t})$ using a gluing diffeomorphism $h\colon\TT^2\to\TT^2$ which identifies the boundaries of $\bar M_f$ and $\bar{\bar M}_f$. The gluing diffeomorphism $h$ is designed so that the week stable foliation of the attractor in $\bar M_f$ is transverse to the weak unstable foliation of the repeller in $\bar{\bar M}_f$ on the gluing torus as indicated in Figure~\ref{fig1}. This control on the foliations on the torus boundary allows to deduce the Anosov property of the surgered flow.

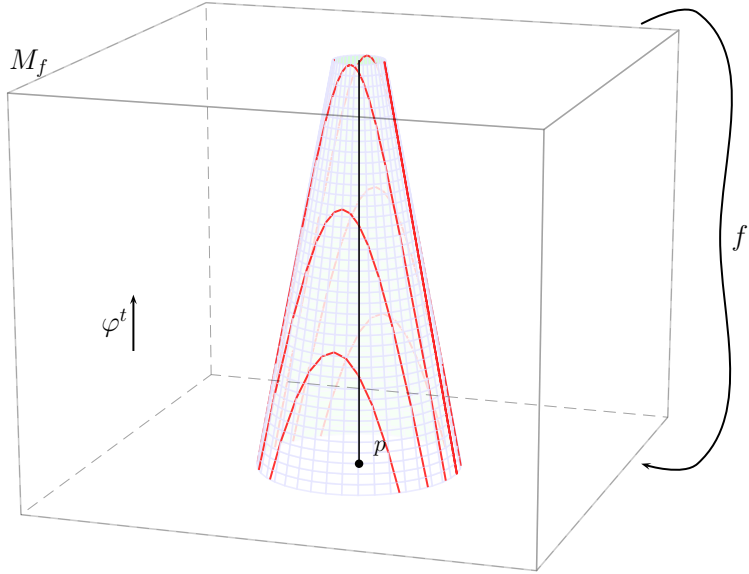
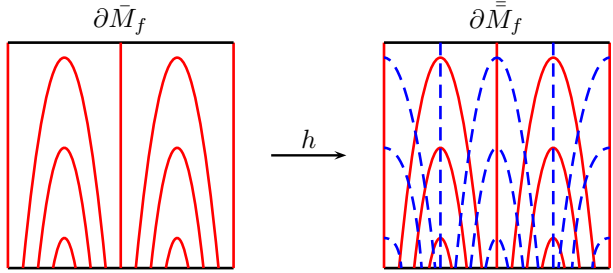
\begin{figure}[h]
\begin{center}
\begin{subfigure}[b]{0.75\textwidth}
 \begin{pspicture}(-5,-1.6)(5.2,6.2)
\psset{viewpoint=60 20 20,Decran=60}
\scalebox{0.75}{
\psline[linewidth=1pt]{->}(-4,2)(-4,3)

\uput{0.1}[180](-4,2.5){\Large$\varphi^t$}

\pscurve[linewidth=1pt]{->}(5,7.8)(6,8)(6.4,4)(6,0)(5,0) %
\put(-6.2,7){\Large$M_f$}
\uput{0.2}[0](6.4,4){\Large$f$}

\psSolid[object=troncconecreux, 
opacity=0.8,
r0=2,r1=.5,h=8,
incolor=green!20,
linecolor=blue!10,
mode=4,
ngrid= 40 40,
intersectiontype=0,
intersectionplan={
[1 0 0 0]
[1 0 0 -0.5]
[1 0 0 -1]
[1 0 0 -1.5]
[1 0 0 1.5]
[1 0 0 0.5]
[1 0 0 1]},
intersectioncolor=(rouge),
intersectionlinewidth= 1]

\psSolid[object=line,
args=0 0 0 0 0 8]

\psPoint(0,0,0){Y}
\uput{0.3}[35](Y){\large$p$}
\psdot[dotsize=0.15](Y)

\psSolid[
object=parallelepiped,
a=10,b=10,c=8,
opacity=0.2,
action=draw,
](0,0,4)}
\end{pspicture} \caption{The suspended DA flow} \label{fig_Franks_Williams_DA}
\end{subfigure}
\begin{subfigure}[b]{0.75\textwidth}
 \scalebox{1}{
  \begin{pspicture}(0,0)(8,4)
\psline[linewidth=1pt](0,0)(3,0)
\psline[linewidth=1pt,linecolor=red](0,0)(0,3)
\psline[linewidth=1pt,linecolor=red](3,0)(3,3)
\psline[linewidth=1pt](0,3)(3,3)
\psline[linewidth=1pt,linecolor=red](1.5,0)(1.5,3)

\parabola[linewidth=1pt,linecolor=red](0.2,0)(0.75,2.8)
\parabola[linewidth=1pt,linecolor=red](0.4,0)(0.75,1.6)
\parabola[linewidth=1pt,linecolor=red](0.6,0)(0.75,0.4)

\rput[bl]{0}(1.5,0){%
\parabola[linewidth=1pt,linecolor=red](0.2,0)(0.75,2.8)
\parabola[linewidth=1pt,linecolor=red](0.4,0)(0.75,1.6)
\parabola[linewidth=1pt,linecolor=red](0.6,0)(0.75,0.4)}

\rput[bl]{0}(5,0){%
\psline[linewidth=1pt](0,0)(3,0)
\psline[linewidth=1pt,linecolor=red](0,0)(0,3)
\psline[linewidth=1pt,linecolor=red](3,0)(3,3)
\psline[linewidth=1pt](0,3)(3,3)
\psline[linewidth=1pt,linecolor=red](1.5,0)(1.5,3)
\parabola[linewidth=1pt,linecolor=red](0.2,0)(0.75,2.8)
\parabola[linewidth=1pt,linecolor=red](0.4,0)(0.75,1.6)
\parabola[linewidth=1pt,linecolor=red](0.6,0)(0.75,0.4)

\rput[bl]{0}(1.5,0){%
\parabola[linewidth=1pt,linecolor=red](0.2,0)(0.75,2.8)
\parabola[linewidth=1pt,linecolor=red](0.4,0)(0.75,1.6)
\parabola[linewidth=1pt,linecolor=red](0.6,0)(0.75,0.4)}

\psbezier[linewidth=1pt,linecolor=blue,linestyle=dashed](0,2.8)(0.4,2.8)(0.67,0.1)(0.68,0)
\psbezier[linewidth=1pt,linecolor=blue,linestyle=dashed](0,1.6)(0.4,1.6)(0.49,0.1)(0.5,0)
\psbezier[linewidth=1pt,linecolor=blue,linestyle=dashed](0,0.4)(0.2,0.4)(0.22,0.1)(0.23,0)
\psline[linewidth=1pt,linecolor=blue,linestyle=dashed](0.75,0)(0.75,3)
\psline[linewidth=1pt,linecolor=blue,linestyle=dashed](2.25,0)(2.25,3)
\rput[bl]{0}(0.75,0){%
\parabola[linewidth=1pt,linecolor=blue,linestyle=dashed](0.2,0)(0.75,2.8)
\parabola[linewidth=1pt,linecolor=blue,linestyle=dashed](0.4,0)(0.75,1.6)
\parabola[linewidth=1pt,linecolor=blue,linestyle=dashed](0.6,0)(0.75,0.4)}
}

\psbezier[linewidth=1pt,linecolor=blue,linestyle=dashed](8,2.8)(7.6,2.8)(7.33,0.1)(7.32,0)
\psbezier[linewidth=1pt,linecolor=blue,linestyle=dashed](8,1.6)(7.6,1.6)(7.51,0.1)(7.5,0)
\psbezier[linewidth=1pt,linecolor=blue,linestyle=dashed](8,0.4)(7.8,0.4)(7.78,0.1)(7.77,0)

\psline[linewidth=.9pt]{->}(3.5,1.5)(4.5,1.5)
\uput{0.1}[90](4,1.5){$h$}
\uput{0.1}[90](1.5,3){$\partial\bar{M}_f$}
\uput{0.1}[90](6.5,3){$\partial\bar{\bar M}_f$}
\end{pspicture}} \caption{Quarter-turn gluing diffeomorphism $h$ which takes solid-line Reeb foliation to the dashed one assures needed transversality.} \label{fig_Quarter-turn}
\end{subfigure}
\end{center}
      \caption{The Franks-Williams construction} \label{fig1}
\end{figure}

For higher dimensional examples Franks and Williams suggested looking at the product $B\times f\colon \TT^d\times\TT^2\to\TT^d\times\TT^2$, where $f$ is the DA diffeomorphism, as  before, and $B$ is a hyperbolic automorphism. The idea is to view the suspension flow of $B\times f$, which we denote by $\varphi^t_B$ as an affine Anosov flow over the suspension flow $\varphi^t$ of $f$ (cf. Example~\ref{ex_suspension2}) and try performing ``twisted double surgery" so that the resulting flow fibers over the 3-dimensional Franks-Williams flow. 

For suspension flow $\varphi^t_B$ the mapping torus $ M_B\times \{p\}$ is a repeller. Removing a neighborhood of $M_B\times \{p\}$ yields a manifold with boundary $M_B\times S^1$  equipped with an inward flow $\varphi^t_B$. The same construction using a hyperbolic automorphism $C\colon \TT^d\to\TT^d$ yields a manifold with boundary $M_C\times S^1$ which we equip with an outward flow $\varphi^{-t}_C$ (\cite[pp.~166-167]{FW} suggest taking $C=B^{-1}$). Now one tries to paste together the flows $\varphi_B^t$ and $\varphi_C^{-t}$ using a gluing $H\colon M_B\times S^1\to M_C\times S^1$ which puts the weak stable distribution of $\varphi_B^t$ transversely to the weak unstable distribution of $\varphi_C^{-t}$. Note that both $\varphi_B^t$ and $\varphi_C^{-t}$ are fiberwise affine flows and both $M_B\times S^1$ and $M_C\times S^1$ are affine $\TT^d$ bundles over $\TT^2$. Hence, we can try to look for a fiberwise affine $H$ which fits into the commutative diagram
$$
\xymatrix{
{M_B\times S^1}\ar[d]\ar^H[r] & M_C\times S^1\ar[d]\\
\TT^2\ar^{h}[r] & \TT^2
}
$$
where $h$ is the ``quarter-turn" gluing diffeomorphism from the Franks-Williams 3-dimensional surgery construction. After cutting $M_C$ and $M_B$ along the transverse tori, the bundles $M_B\times S^1\to\TT^2$ and $M_C\times S^1\to\TT^2$ both become trivial bundles $\TT^d\times[0,1]\times S^1\to [0,1]\times S^1$, hence $H$ must have the form
\begin{equation}
\label{eq_H}
H(x,y)=(Dx, h(y)),
\end{equation}
where $D\colon\TT^d\to\TT^d$ is an automorphism. Because formula~(\ref{eq_H}) must induce a diffeomorphism of mapping tori we have the following condition
\begin{equation}
\label{eq_BC}
C=DBD^{-1}.
\end{equation}
This implies that the suspension flows $\varphi^t_B$ and $\varphi^t_C$ are conjugate. Therefore, if we look for $H$ in the affine form~(\ref{eq_H}) then the stable distributions of $\varphi_B^t$ and $\varphi_C^t$ in the fiber are identified via $H$ along the boundary. Hence, the stable distribution in the fiber of $\varphi_B^t$ is identified with the unstable distribution in the fiber of $\varphi_C^{-t}$ via $H$. We conclude that the sought transversality cannot be achieved in this way.
\begin{rema} We do not know if one can repair this construction by considering more general glueing  diffeomorphisms $H\colon M_B\times S^1\to M_C\times S^1$. Note that~(\ref{eq_BC}) is a condition for manifolds $M_B$ and $M_C$ to be homeomorphic.
\end{rema}

Note also that Theorem~\ref{thm1} yields the following corollary.

\begin{coro}
Let $\varphi^t$ be a 3-dimensional non-transitive Anosov flow. Assume that $\Phi^t\colon E\to E$ is a fiberwise Anosov flow over $\varphi^t$. Then $s=u\ge 3$. (Here $s$ and $u$ are dimensions of stable and unstable distributions of $\Phi^t$.)
\end{coro}

\subsection{Proof of Theorem~\ref{thm_fiberwise_anosov_over_n-dimensional_anosov}}

Let $a$ and $b$ be the periodic orbits of $\varphi^t$ such that $a$ is freely homotopic to $-b$. 

Let $H\colon [0,1]\times S^1\to M$ be an homotopy between $a$ and $-b$. We pull-back $E\to M$ using $H$ to obtain an affine $\TT^d$ bundle $H^*E\to[0,1]\times S^1$. Up to translation on $\TT^d$, the monodromy of this bundle is given by an automorphism $A\colon \TT^d\to\TT^d$. Because the boundary of the base cylinder corresponds to the pair of periodic orbits $a$ and $-b$ we also have that the boundary of $H^*E$ is equipped with the  Anosov flow. It follows that the first return map of the flow to $\TT^d$ along one boundary component, say $\{0\}\times S^1$, is an Anosov diffeomorphism $f\colon \TT^d\to\TT^d$ homotopic to  $A$. And the first return map of the flow to $\TT^d$ along the other boundary component, $\{1\}\times S^1$, is an Anosov diffeomorphism $g\colon \TT^d\to\TT^d$ homotopic to  $A^{-1}$. 
 By work of Franks and Manning~\cite{Fr, M}, the automorphism $A$ is hyperbolic and $f$ is conjugate to $A$. Hence, the dimension of the stable subspace of $A$ must be $\dim V^s$ --- the dimension of the stable bundle of $f$. Similarly, the dimension of the stable subspace of $A^{-1}$ (= dimension of the unstable subspace of $A$) is $\dim V^s$ --- the dimension of the stable bundle of $g$. Hence $d$ is even and the fiberwise stable and unstable distributions have same dimension.

To finish the proof it remains to rule out the case when $d=2$. 
To do that we will need the following simple lemma.
\begin{lemm} \label{lem_no_circle_leaves}
Let $\TT^2\to E\to M$ be an affine bundle and let $\Phi^t\colon E\to E$ be a fiberwise Anosov flow. Denote by $V^u\oplus V^s$ the vertical Anosov splitting and by $\mathcal V^u$ the integral foliation of $V^u$. Then $\mathcal V^u$ does not have circle leaves.
\end{lemm}
\begin{proof} Indeed, if there is a circle $\mathcal C$ tangent to $V^u$ then the length of $\Phi^{-t}(C)$ will go to zero as $t\to\infty$ by the Anosov property. If $\mathcal C$ is contractible then this is impossible because $V^u$ is uniformly continuous. And if $\mathcal C$ is homotopically non-trivial this is also impossible because there exists a lower bound on the length of homotopically non-trivial loops due to compactness of the total space $E$.
\end{proof}

The above lemma implies that, on every torus fiber, the foliation $\mathcal V^u$ does not have Reeb components.
We say that a foliation of the torus $\TT^2$ is {\it parallel to a linear foliation} if the lifts of its leaves to the universal cover $\RR^2$ remain at finite distance from leaves of a linear foliation on $\RR^2$.  
Every $\mathcal V$ on $\TT^2$ with no Reeb components is always parallel to exactly one linear foliation $\mathcal W$. This follows from the fact that the rotation number is well-defined for Reebless foliations. Further, if one varies $\mathcal V$ continuously the corresponding linear foliation $\mathcal W$ also varies continuously.

Assume that there exists a fiberwise Anosov flow $\Phi^t\colon E\to E$, with $\TT^2$ fibers, which fibers over an Anosov flow $\varphi^t$.
Consider the fiberwise Anosov splitting $VE=V^s\oplus V^u$, and the corresponding integral foliations, whose leaves are contained in the $\TT^2$ fibers. Pulling back the ``vertical unstable foliation" to the affine bundle $\TT^2\to H^*E\to[0,1]\times S^1$ and then restricting to the segment $[0,1]\times p\subset [0,1]\times S^1$ yields a continuous one-parameter family of foliations $\mathcal V_t$, $t\in [0,1]$. The foliation $\mathcal V_0$ is the unstable foliation of $f$ and $\mathcal V_1$ is the unstable foliation of $g$ (where $f$ and $g$ are described at the beginning of the proof). For each $t\in[0,1]$, the foliation $\mathcal V_t$ is parallel to a unique linear foliation $\mathcal W_t$ which varies continuously with $t$. Because $f$ is conjugate to $A$, $\mathcal V_0$ is parallel to the linear foliation $\mathcal W_0$ given by the unstable eigendirection of $A$. Similarly, $\mathcal V_1$ is parallel to the linear foliation $\mathcal W_1$ given by the stable eigendirection of $A$. 

Because $\mathcal W_0$ and $\mathcal W_1$ are different linear irrational foliations, the intermediate value theorem yields a $t_0$ such that $\mathcal W_{t_0}$ is a rational foliation. Because $\mathcal V_{t_0}$ is parallel to $\mathcal W_{t_0}$, $\mathcal V_{t_0}$ has rational rotation number, \ie the return map to a circle transversal to $\mathcal V_{t_0}$ has rational rotation number. It follows that $\mathcal V_{t_0}$ has a circle leaf, which is in contradiction with Lemma \ref{lem_no_circle_leaves}.

\subsection{Proof of Theorem~\ref{thm1}}

%

In order to prove Theorem~\ref{thm1}, all we have to do is show that, if $\Phi^t$ is a fiberwise Anosov flow over a $3$-dimensional Anosov flow $\varphi^t$, then Theorem \ref{thm_fiberwise_anosov_over_n-dimensional_anosov} applies except when $\varphi^t$ is orbit equivalent to a suspension, which, together with Proposition \ref{prop_suspensions}, yields Theorem \ref{thm1}. This characterization of suspensions in $3$-manifold follows directly from works of Barbot and Fenley. However, as we did not find this property explicitly stated in the way we needed in their work, we provide a proof. 

\begin{theo}[Barbot, Fenley]\label{t.free} 
Let $\varphi^t\colon M\to M$ be an Anosov flow on a closed $3$-manifold $M$. Then one of the following assertions holds:
\begin{enumerate}
 \item either $\varphi^t$ is topologically equivalent to the suspension flow of an Anosov automorphism of the torus $\TT^2$;
 \item or there exist periodic orbits $a$ and $b$ such that $a$ is freely homotopic to $-b$, \ie to the orbit $b$ with reversed orientation. 
\end{enumerate}
\end{theo}

\begin{rema}
 Theorem~\ref{t.free} is essentially given in Corollary~E of \cite{Fen:SBAF}, but the author does not state the fact that the free homotopy can be between two orbits with reversed orientation, so we recall the arguments.
\end{rema}

\begin{rema}
 In the statement of Theorem~\ref{t.free}, $a$ and $b$ might have to be taken as going twice around the periodic orbit of $\flot$. Another possibility is that $a$ and $b$ are actually the same, \ie $a$ is freely homotopic to its inverse.
\end{rema}

We proceed with some preparations for the proof. A pair of periodic orbits which are freely homotopic to the inverse of each other will be found as corners of lozenges that we will now define.

If $\flot$ is an Anosov flow on a $3$-manifold $M$, then we denote by $\mathcal{F}^{ss}$, $\mathcal{F}^{uu}$, $\fs$ and $\fu$ the strong stable, strong unstable, stable and unstable foliations, respectively.
All of these foliations, as well as the flow, lift to the universal cover $\wt M$ of $M$, and we denote the lifts by $\hflot$, $\hfs$, $\hfu$, $\wt{\mathcal{F}}^{ss}$ and $\wt{\mathcal{F}}^{uu}$.
It is well known that $\hfs$ and $\hfu$ are foliations by planes, and Palmeira's theorem then implies that $\wt M$ is homeomorphic to $\RR^3$.

We define the \emph{orbit space} of $\flot$ as $\wt M$ quotiented out by the relation ``being on the
same orbit of $\hflot$ '', and we denote the orbit space by $\orb$.
We also define the \emph{stable} (resp. \emph{unstable}) \emph{leaf space} of $\flot$ as the quotient of 
$\wt M$ by the relation ``being on the same leaf of $\hfs$ (resp. $\hfu$)''. We denote these leaf spaces by $\leafs$ and $\leafu$, respectively.

In general, the stable and unstable leaf spaces are non-Hausdorff. However, the orbit space $\orb$ of an Anosov flow in dimension $3$ is always homeomorphic to $\RR^2$ (see~\cite[Th\'eor\`eme 3.1]{Bar:CFA} or~\cite[Proposition 2.1]{Fen:AFM}). The foliations 
$\hfs$ and $\hfu$ descend to foliations by lines of $\orb$ that we still denote by $\hfs$ and $\hfu$.
When one leaf space is Hausdorff, then the other is also Hausdorff (see~\cite[Th\'eor\`eme 4.1]{Bar:CFA} or~\cite[Theorem 3.4]{Fen:AFM}), and, in this case, we say that the flow is $\RR$-\emph{covered}.

The following notion of lozenge proved to be essential in the study of Anosov flow in $3$-manifolds.

\begin{defi}
Let $\alpha,\beta$ be two orbits in $\orb$ and let $A \subset \hfs(\alpha)$, $B \subset \hfu(\alpha)$, $C \subset \hfs(\beta)$ and $D \subset \hfu(\beta)$  be four half leaves satisfying:
\begin{itemize}
 \item For any $\lambda^s \in \leafs$, $\lambda^s \cap B \neq \emptyset$ if and only if $\lambda^s \cap D\neq \emptyset$,
 \item For any $\lambda^u \in \leafu$, $\lambda^u \cap A \neq \emptyset$ if and only if $\lambda^u \cap C \neq \emptyset$,
 \item The half-leaf $A$ does not intersect $D$ and $B$ does not intersect $C$.
\end{itemize}

A \emph{lozenge $L$ with corners $\alpha$ and $\beta$} is the open subset of $\orb$ given by (see Figure~\ref{fig:a_lozenge})
\begin{equation*}
 L := \lbrace p \in \orb \mid \hfs(p) \cap B \neq \emptyset, \; \hfu(p) \cap A \neq \emptyset \rbrace.
\end{equation*}
The half-leaves $A,B,C$ and $D$ are called the \emph{sides}.
\end{defi}

\begin{figure}[h]
\begin{center}
\scalebox{0.85}{   
\begin{pspicture}(0,-1.97)(3.92,1.97)
\psbezier[linewidth=0.04](0.02,0.17)(0.88,0.03)(0.74114233,-0.47874346)(1.48,-1.11)(2.2188578,-1.7412565)(2.38,-1.73)(3.26,-1.95)
\psbezier[linewidth=0.04](0.0,0.27)(0.96,0.3105634)(0.75286174,0.37057108)(1.78,0.77)(2.8071382,1.169429)(2.66,1.47)(3.36,1.45)
\psbezier[linewidth=0.04](0.6,1.95)(1.2539726,1.8871263)(1.1265805,1.3646309)(2.0345206,0.7973154)(2.9424605,0.23)(3.2249315,0.2543258)(3.9,0.31)
\psbezier[linewidth=0.04](0.52,-1.33)(1.48,-1.33)(1.3597014,-0.9703507)(2.3,-0.63)(3.2402985,-0.28964934)(3.14,0.05)(3.84,0.23)
\psdots[dotsize=0.16](1.98,0.85)
\psdots[dotsize=0.16](1.46,-1.11)
\usefont{T1}{ptm}{m}{n}
\rput(1.6145313,-0.06){$L$}
\usefont{T1}{ptm}{m}{n}
\rput(1.4,-1.38){$\wt\alpha$}
\rput(2,1.2){$\wt\beta$}

\rput(0.6,-0.6){$A$}
\rput(3,-0.6){$B$}
\rput(0.6,0.6){$D$}
\rput(3,0.6){$C$}
\end{pspicture}
}
\end{center}
\caption{A lozenge with corners $\wt\alpha$, $\wt\beta$ and sides $A,B,C,D$} \label{fig:a_lozenge}
\end{figure}
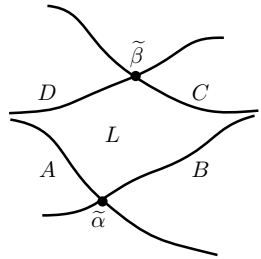

\begin{lemm} \label{lem:periodic_lozenges}
 Let $L$ be a lozenge with corners $\alpha$ and $\beta$. Suppose that $\alpha$ corresponds to a lift of a periodic orbit $\gamma_{\alpha}$.
Then $\beta$ corresponds to a lift of a periodic orbit $\gamma_{\beta}$, and $\gamma_{\alpha}$ (or, possibly, the orbit obtained by going twice around $\gamma_{\alpha}$) is freely homotopic to the inverse of $\gamma_{\beta}$ (or, possibly, the orbit obtained by going twice around $\gamma_{\beta}$).
\end{lemm}

\begin{rema}
 The proof of this result can be found in~\cite[Section 3]{Fen:QGAF} or in~\cite[Section 3]{Bar:MPOT}, but it is not stated exactly as we want, so we provide the argument.
\end{rema}

\begin{proof}
 Since $\alpha$ is the image in $\orb$ of a lift of a periodic orbit $\gamma_{\alpha}$, there exists an element $g\in \pi_1(M)$ which represents this orbit $\gamma_{\alpha}$ (oriented according to the flow direction) and which leaves $\alpha$ invariant.
Note that, since $g$ represents $\gamma_{\alpha}$, the action of $g$ on the unstable leaf $\hfu(\alpha)$ in the orbit space $\orb$ is a contraction, and its action on $\hfs(\alpha)$ is an expansion.

Since $g$ fixes $\hfu(\alpha)$, either $g$ fixes the side of $L$ on $\hfu(\alpha)$, or $g^2$ does. And similarly for the side of $L$ in  $\hfs(\alpha)$. So, up to considering the orbit that goes twice around $\gamma_{\alpha}$, and taking $g^2$ instead of $g$, we can assume that $g$ represents $\gamma_{\alpha}$ and that it fixes the sides of $L$ on the stable and unstable leaves of $\alpha$ in $\orb$.

By the definition of lozenge, since $g$ fixes the sides of $L$ that contain $\alpha$, it also should fix the opposite sides of $L$ (see, for instance, \cite[Lemme 3.4]{Bar:MPOT}). Hence, the opposite corner $\beta$ is the image in $\orb$ of a lift of a periodic orbit $\gamma_{\beta}$ of $\flot$.

Because the action of $g$ on $\hfu(\alpha)$ in $\orb$ is a contraction, its action on $\hfu(\beta)$ in $\orb$ has to be an expansion. Otherwise, there would need to be at least one stable leaf intersecting both $\hfu(\alpha)$ and $\hfu(\beta)$ fixed by $g$. That would imply the existence of two periodic orbits on, for instance, the unstable leaf of $\gamma_{\alpha}$, which is impossible.

Hence, if $h$ is the element of $\pi_1(M)$ that represents $\gamma_{\beta}$ and fixes $\beta$, there exists $n \geq 1$ such that $h^{n}= g^{-1}$. But, by arguing in the same way as before we can shows that $h^2$ has to fix $\alpha$. Hence $h= g^{-1}$ or $h^{2}= g^{-1}$.

If $h= g^{-1}$, then $\gamma_{\beta}$ is freely homotopic to the inverse of $\gamma_{\alpha}$. And if $h^{2}= g^{-1}$, then the orbit of the flow obtained by going twice around $\gamma_{\beta}$ is freely homotopic to the inverse of $\gamma_{\alpha}$.
\end{proof}

\begin{proof} [Proof of Theorem~\ref{t.free}]
 If $\flot$ is $\RR$-covered, then  $\flot$ is transitive~\cite[Th\'eor\`emes 2.7]{Bar:CFA}. Further, if every leave of $\widetilde{\mathcal F}^{uu}$ intersects every leave of $\hfs$, then $\flot$ is a suspension of an Anosov diffeomorphism~\cite[Th\'eor\`emes 2.7]{Bar:CFA}.

If $\flot$ is $\RR$-covered and not a suspension of an Anosov diffeomorphism,  then $\flot$ is \emph{skewed}~\cite[Theorem 3.4]{Fen:AFM}, and hence every lift of every periodic orbit is a corner of a lozenge (see, for instance, the proof of \cite[Theorem 3.4]{Fen:AFM}). Then Lemma~\ref{lem:periodic_lozenges} concludes the proof.

If $\flot$ is not $\mathbb{R}$-covered, then there exist lozenges whose corners correspond to periodic orbits~\cite[Theorem B and Corollary 4.4]{Fen:SBAF}. In this case, again, the proof is complete by Lemma \ref{lem:periodic_lozenges}.
\end{proof}

\subsection{Proof of Corollary \ref{cor_algebraic_flows}}

Let us recall Tomter classification result (see \cite[Theorem 4]{BM13}): Let $\Phi^t$ be an algebraic Anosov flow. Then, up to passing to a finite cover, $\Phi^t$ is either a suspension of an Anosov automorphism of a nilmanifold, or it is a \emph{nil-suspension} over the geodesic flow of a locally symmetric space of real rank one.
A nil-suspension means that $\Phi^t\colon E \to E$ is an affine fiberwise Anosov flow over a geodesic flow $\varphi^t\colon M \to M$, as in Definition \ref{def_anosov}, except that $E$ is now a fiber bundle over $M$ whose fiber is any \emph{nilmanifold} $N$, instead of just a torus.

So Corollary \ref{cor_algebraic_flows} is a direct consequence of Theorem \ref{thm_fiberwise_anosov_over_n-dimensional_anosov} when $\Phi^t$ is a $\TT^d$-suspension over a geodesic flow (i.e., a fiberwise Anosov flow in the setting of Definition \ref{def_anosov}). To deal with the more general nil-suspensions, one just have to recall that every nil-suspensions can be obtained by successive $\TT^d$-suspensions (see, \cite[Section 3.3.1]{BM13}). Hence, applying Theorem \ref{thm_fiberwise_anosov_over_n-dimensional_anosov} iteratively to each of the successive suspensions yields Corollary \ref{cor_algebraic_flows}.

\section{Proof of Theorem \ref{thm2}} \label{sec_3}

\subsection{Hyperbolic manifolds which admit totally geodesic hypersurfaces} \label{subsec_total_geod_hypersurface}

Denote by $M$ an $n$-dimensional closed oriented hyperbolic manifold \ie closed manifold modeled on the hyperbolic $n$-space $\mathbb H^n$, $n\ge 2$. Let $N$ be a closed totally geodesic codimension one submanifold in $M$. Further we will assume that the normal bundle of $N$ is trivial, \ie that $N\subset M$ is two-sided. We will call such $N\subset M$ a {\it totally geodesic hypersurface} in $M$. Clearly $N$ is $(n-1)$-dimensional hyperbolic manifold. Denote by $SM$ and $SN$ the unit tangent bundles of $M$ and $N$, respectively. Then the induced embedding $SN\subset SM$ is codimension two and, clearly, the geodesic flow on $SM$ leaves $SN$ invariant.

When $n=2$ the hypersurface $N$ is just a simple closed geodesic. For $n\ge 3$ hyperbolic manifolds which admit two-sided totally geodesic hypersurfaces can also be constructed.
When  $n=3$ hyperbolic manifolds which admit such totally geodesic subsurfaces (or, equivalently, cocompact torsion-free Kleinian groups which admit cocompact Fuchsian subgroups) can be constructed in multiple ways. One elementary method to construct such a pair $N\subset M$ is by gluing $M$ from a compact hyperbolic tetrahedra whose dihedral angles are submultiples of $\pi$~\cite{BM}. (F.~Lann\'er showed that there are precisely 9 such tetrahedra~\cite{L}.) The Kleinian group generated by reflection in the faces of such tetrahedron is a Coxeter group whose index two subgroup consisting of orientation preserving isometries is the fundamental group of sought manifold $M$. Clearly each face develops into a closed totally geodesic surface. In general, this surface is only immersed, but it becomes embedded (and two-sided) after passing to a finite cover.

For constructions in all dimensions $\ge 3$ see~\cite{GPS} and~\cite[Section 5]{CD}.

\subsection{Special coordinates adapted to the geodesic flow around the totally geodesic hypersurface}

In this section we construct a special parametrization of a neighborhood of $SN\subset SM$ which will be well suited to performing the pitchfork bifurcation on $SN$ in Section \ref{sec_DA_flow}.

\subsubsection{A 2-dimensional submanifold of $S\bH^2$} \label{subsec_special_subset}
We start by describing some special submanifold in $S\bH^2$, the unit tangent bundle of the hyperbolic disc.

Fix $(x,v)$ in $S\bH^2$. Let $c_v(t)$ be the geodesic through $x$ in the direction $v$. Denote by $v^+$ and $v^-$ the endpoints of $c_v(t)$ in $\partial_{\infty}\bH^2$. Choose $u \in S_x\bH^2$ orthogonal to $v$, and denote by $u_t \in S_{c_v(t)}\bH^2$ the vector obtained by parallel transport of $u$ along $c_v$.
We consider the following curves in $\bH^2$ which pass through $x$.
\begin{enumerate}
 \item $S(s)$ is the horocycle through $x$ and $v^+$, and $U(s)$ is the horocycle through $x$ and $v^-$; we parametrize horocycles by arc-length.
 \item For any $t\in \R$, let $\alpha_t(s)$ be the circle centered at $c_v(t)$ and passing through $x$.
 \item For any $t\in \R$, let $k_t(s)$ be the geodesic through $c_v(t)$ in the direction $u_t$ and let $\gamma_t(s)$ the equidistant curve to $k_t(s)$ passing through $x$. We choose the parametrization of $\gamma_t$ so that, for any $s$, $d(\gamma_t(s), k_t(s)) = t$.
\end{enumerate}

\begin{definition}\label{def_special_subset}
For any $(x,v)\in \bH^2$, we define the set $ U_{(x,v)} \subset S\bH^2$ by
\begin{equation*}
 (y,w) \in U_{(x,v)} \quad \text{if and only if } \left\{\begin{aligned}
                                                   y &= x \text{ and } w \in S_x\bH^2\\
                                                   y &= S(s) \text{ and } w = \dot{S}(s) - \pi/2 \\
                                                   y &= U(s) \text{ and } w = \dot{U}(s) - \pi/2 \\
                                                   y &= \alpha_t(s) \text{ and } w = \dot{\alpha_t}(s) - \pi/2 \\
                                                   y &= \gamma_t(s) \text{ and } w = \dot{\gamma_t}(s) - \pi/2 
                                                  \end{aligned} \right.
\end{equation*}
where $w= w' - \pi/2$ means that $w$ was obtained by rotation of $\pi/2$ from $w'$ and the orientation is chosen so that $v = u - \pi/2$ (see Figure \ref{fig_special_neighborhood}).
\end{definition}

  \begin{figure}[h]
   \begin{center}
   \scalebox{0.7}{ 
   \begin{pspicture}(-5,-5)(5,5)
 \pscircle[linewidth=1pt](0,0){5}
 \psline[linewidth=1pt](-5,0)(5,0)
 \psline[linewidth=1pt,linecolor=red](0,-5)(0,5)
 \pscircle[linewidth=1pt,linecolor=blue](2.5,0){2.5}
 \pscircle[linewidth=1pt,linecolor=blue](-2.5,0){2.5}
 \pscircle[linewidth=1pt,linecolor=green](2,0){2}
 \pscircle[linewidth=1pt,linecolor=green](-2,0){2}
 \pscircle[linewidth=1pt,linecolor=green](-1,0){1}
 \pscircle[linewidth=1pt,linecolor=green](1,0){1}
 \psarc[linewidth=1pt,linecolor=red](5,0){5}{120}{240}
 \rput[l]{180}{\psarc[linewidth=1pt,linecolor=red](5,0){5}{120}{240}}
 \psline[linewidth=1pt]{->}(0,0)(0.25,0)
 \psrotate(0,0){-55}{\psline[linewidth=1pt]{->}(0,0)(0.25,0)}
 \psrotate(0,0){55}{\psline[linewidth=1pt]{->}(0,0)(0.25,0)}
 \uput{1.5pt}[50](0.25,0){$v$}
 \psrotate(2,0){-45}{\psline[linewidth=1pt,linecolor=green]{->}(0,0)(0.25,0)}
 \psrotate(2,0){45}{\psline[linewidth=1pt,linecolor=green]{->}(0,0)(0.25,0)}
 \psrotate(2,0){90}{\psline[linewidth=1pt,linecolor=green]{->}(0,0)(0.25,0)}
 \psrotate(2,0){-90}{\psline[linewidth=1pt,linecolor=green]{->}(0,0)(0.25,0)}
 \psrotate(2,0){135}{\psline[linewidth=1pt,linecolor=green]{->}(0,0)(0.25,0)}
 \psrotate(2,0){-135}{\psline[linewidth=1pt,linecolor=green]{->}(0,0)(0.25,0)}
 \psrotate(-2,0){-45}{\psline[linewidth=1pt,linecolor=green]{->}(0,0)(0.25,0)}
 \psrotate(-2,0){45}{\psline[linewidth=1pt,linecolor=green]{->}(0,0)(0.25,0)}
 \psrotate(-2,0){90}{\psline[linewidth=1pt,linecolor=green]{->}(0,0)(0.25,0)}
 \psrotate(-2,0){-90}{\psline[linewidth=1pt,linecolor=green]{->}(0,0)(0.25,0)}
 \psrotate(-2,0){135}{\psline[linewidth=1pt,linecolor=green]{->}(0,0)(0.25,0)}
 \psrotate(-2,0){-135}{\psline[linewidth=1pt,linecolor=green]{->}(0,0)(0.25,0)}
 \psrotate(1,0){-45}{\psline[linewidth=1pt,linecolor=green]{->}(0,0)(0.25,0)}
 \psrotate(1,0){45}{\psline[linewidth=1pt,linecolor=green]{->}(0,0)(0.25,0)}
 \psrotate(1,0){85}{\psline[linewidth=1pt,linecolor=green]{->}(0,0)(0.25,0)}
 \psrotate(1,0){-85}{\psline[linewidth=1pt,linecolor=green]{->}(0,0)(0.25,0)} 
 \psrotate(1,0){125}{\psline[linewidth=1pt,linecolor=green]{->}(0,0)(0.25,0)}
 \psrotate(1,0){-125}{\psline[linewidth=1pt,linecolor=green]{->}(0,0)(0.25,0)} 
 \psrotate(-1,0){-45}{\psline[linewidth=1pt,linecolor=green]{->}(0,0)(0.25,0)}
 \psrotate(-1,0){45}{\psline[linewidth=1pt,linecolor=green]{->}(0,0)(0.25,0)}
 \psrotate(-1,0){85}{\psline[linewidth=1pt,linecolor=green]{->}(0,0)(0.25,0)}
 \psrotate(-1,0){-85}{\psline[linewidth=1pt,linecolor=green]{->}(0,0)(0.25,0)} 
 \psrotate(-1,0){125}{\psline[linewidth=1pt,linecolor=green]{->}(0,0)(0.25,0)}
 \psrotate(-1,0){-125}{\psline[linewidth=1pt,linecolor=green]{->}(0,0)(0.25,0)} 
 \psrotate(2.5,0){-45}{\psline[linewidth=1pt,linecolor=blue]{->}(0,0)(0.25,0)}
 \psrotate(2.5,0){45}{\psline[linewidth=1pt,linecolor=blue]{->}(0,0)(0.25,0)}
 \psrotate(2.5,0){75}{\psline[linewidth=1pt,linecolor=blue]{->}(0,0)(0.25,0)}
 \psrotate(2.5,0){-75}{\psline[linewidth=1pt,linecolor=blue]{->}(0,0)(0.25,0)}
  \psrotate(2.5,0){105}{\psline[linewidth=1pt,linecolor=blue]{->}(0,0)(0.25,0)}
 \psrotate(2.5,0){-105}{\psline[linewidth=1pt,linecolor=blue]{->}(0,0)(0.25,0)}
 \psrotate(-2.5,0){-45}{\psline[linewidth=1pt,linecolor=blue]{->}(0,0)(0.25,0)}
 \psrotate(-2.5,0){45}{\psline[linewidth=1pt,linecolor=blue]{->}(0,0)(0.25,0)}
 \psrotate(-2.5,0){75}{\psline[linewidth=1pt,linecolor=blue]{->}(0,0)(0.25,0)}
 \psrotate(-2.5,0){-75}{\psline[linewidth=1pt,linecolor=blue]{->}(0,0)(0.25,0)}
 \psrotate(-2.5,0){105}{\psline[linewidth=1pt,linecolor=blue]{->}(0,0)(0.25,0)}
 \psrotate(-2.5,0){-105}{\psline[linewidth=1pt,linecolor=blue]{->}(0,0)(0.25,0)}
 \psrotate(-5,0){-35}{\psline[linewidth=1pt,linecolor=red]{->}(0,0)(0.25,0)}
 \psrotate(-5,0){35}{\psline[linewidth=1pt,linecolor=red]{->}(0,0)(0.25,0)}
 \psrotate(-5,0){50}{\psline[linewidth=1pt,linecolor=red]{->}(0,0)(0.25,0)}
 \psrotate(-5,0){-50}{\psline[linewidth=1pt,linecolor=red]{->}(0,0)(0.25,0)}
 \psrotate(-5,0){15}{\psline[linewidth=1pt,linecolor=red]{->}(0,0)(0.25,0)}
 \psrotate(-5,0){-15}{\psline[linewidth=1pt,linecolor=red]{->}(0,0)(0.25,0)}
 \psrotate(5,0){-35}{\psline[linewidth=1pt,linecolor=red]{->}(0,0)(0.25,0)}
 \psrotate(5,0){35}{\psline[linewidth=1pt,linecolor=red]{->}(0,0)(0.25,0)}
 \psrotate(5,0){50}{\psline[linewidth=1pt,linecolor=red]{->}(0,0)(0.25,0)}
 \psrotate(5,0){-50}{\psline[linewidth=1pt,linecolor=red]{->}(0,0)(0.25,0)}
 \psrotate(5,0){15}{\psline[linewidth=1pt,linecolor=red]{->}(0,0)(0.25,0)}
 \psrotate(5,0){-15}{\psline[linewidth=1pt,linecolor=red]{->}(0,0)(0.25,0)}
  \psline[linewidth=1pt,linecolor=red]{->}(0,2)(0.25,2)
  \psline[linewidth=1pt,linecolor=red]{->}(0,3)(0.25,3)
  \psline[linewidth=1pt,linecolor=red]{->}(0,4)(0.25,4)
  \psline[linewidth=1pt,linecolor=red]{->}(0,-2)(0.25,-2)
  \psline[linewidth=1pt,linecolor=red]{->}(0,-3)(0.25,-3)
  \psline[linewidth=1pt,linecolor=red]{->}(0,-4)(0.25,-4)
\end{pspicture}}
   \end{center}
   \caption{The set $U_{(x,v)}$ with the five different types of elements} \label{fig_special_neighborhood}
  \end{figure}
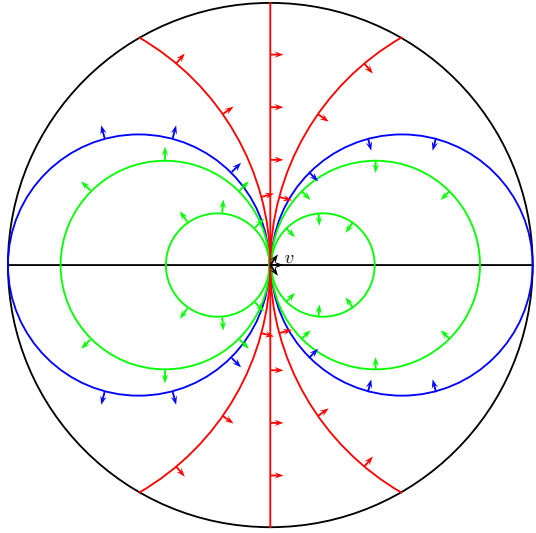

An essential property of the set $U_{(x,v)}$ is that it is invariant under the geodesic flow in the following sense.

\begin{proposition}\label{properties_of_special_set}
For all $(x,v)\in S\bH^2$, and $\eps>0$, we define 
  \[
   U_{(x,v),\eps} = \left\{ (y,w) \in U_{(x,v)} \mid d_{\mathrm{Sas}}((y,w),\dot{c}_v) \le \eps \right\},
  \]
where $d_{\mathrm{Sas}}$ denotes the Sasaki distance induced by the Sasaki Riemannian metric on $S\bH^2$.
  Then the sets $U_{(x,v)}$ and $U_{(x,v),\eps}$, $(x,v)\in S\bH^2$ satisfy the following:
 \begin{enumerate}
  \item \label{flow_invariance} For all $(x,v)\in S\bH^2$ and $t \in \R$ we have
  \[
   g^t\left( U_{(x,v)}\right) = U_{g^t(x,v)},
  \]
  where $g^t$ is the geodesic flow.
  \item \label{flip_invariance} For all $(x,v)\in S\bH^2$ and $\eps>0$, 
  \[
   \sigma \left(U_{(x,v),\eps}\right) = U_{(x,-v),\eps},
  \]
where $\sigma \colon S\bH^2 \rightarrow S\bH^2$ is the flip map, i.e., $\sigma (y,w) = (y, -w)$.
 \item \label{disc} For all $(x,v)\in S\bH^2$, and $\eps>0$ small enough ($\eps < \pi$), the set $U_{(x,v),\eps}$ is a $2$-dimensional smoothly embedded disc which contains $(x,v)$.
 \end{enumerate}
\end{proposition}

\begin{remark}
 In our definition of $U_{(x,v),\eps}$, we choose to take the intersection of $U_{(x,v)}$ with the $\eps$-tubular neighborhood (with respect to the Sasaki metric) of the geodesic $\dot c_v \subset S\bH^2$. We would like to emphasize that it is not the same thing as taking all points inside $U_{(x,v)}$ that are at distance $\eps$ \emph{from $(x,v)$}. Indeed, of all the radial curves that we use in our definition to foliate $U_{(x,v)}$, only four of them are geodesics of the Sasaki metric: the two horospheres, the fiber over $x$, and the curve projecting to $\gamma_0(s)=k_0(s)$. For all the other curves, their \emph{projection to $\bH^2$} are also projections of Sasaki geodesics, but their ``vector part" makes them fail to be Sasaki geodesics (see the classification of Sasaki geodesics in \cite{Sasaki76}). Hence, the orthogonal projection of most points in $U_{(x,v)}$ to $\dot c_v$ will not be $(x,v)$.
\end{remark}

\begin{proof}
 Part \ref{flow_invariance} is a straightforward consequence of our definitions and of hyperbolic geometry: one just have to verify that the five types of points $(y,w) \in U_{(x,v)}$ flows to one of the five types in $U_{g^t(x,v)}$. This is trivial if $(y,w)=(x,v)$. If $y \in U(s)$ or $y\in S(s)$, then, according to our definition, $(y,w)$ is in the strong stable or strong unstable leaf through $(x,v)$, which are invariant under the flow.
 If $y\in \alpha_{t'}(s)$, for some $t'$, then the projection of $g^t(y,w)$ onto $\bH^2$ is a circle centered at $c_v(t')= c_{g^t(x,v)}(t-t')$ of radius $|t'-t|$, and the vector direction of $g^t(y,w)$ is still orthogonal to that circle. Hence, $g^t(y,w)\in  U_{g^t(x,v)}$. A similar argument deals with the last case.
 
 Part~\ref{flip_invariance} follows from the fact that, by our definition, $\sigma \left(U_{(x,v)}\right) = U_{(x,-v)}$, and the fact that $\sigma$ is an isometry with respect to the Sasaki metric.
 
%
%
We are left with showing that part~\ref{disc} holds.
For any $(x,v)\in S\bH^2$, let $P_{(x,v)}$ be the plane distribution in $TS\bH^2$ orthogonal to $X(x,v)$ for the Sasaki metric, where $X\colon S\bH^2 \to TS\bH^2$ is the vector field generating the geodesic flow. Then, by the classification of geodesics of the Sasaki metric \cite{Sasaki76}, the five family of curves used for the definition of $U_{(x,v)}$ are all projections of Sasaki geodesics. Moreover, one has that $\pi(U_{(x,v)})= \pi\left(\exp_{\mathrm{Sas},(x,v)}P_{(x,v)}\right)$, where $\pi \colon S\bH^2 \rightarrow \bH^2$ is the projection and, $\exp_{\mathrm{Sas}}$ is the exponential map of the Sasaki metric. In particular, we have that $U_{(x,v)}$ is obtained from $\exp_{\mathrm{Sas},(x,v)}P_{(x,v)}$ by rotating the vector directions to make them orthogonal to the projections of the Sasaki geodesics (and the amount one needs to rotate depends smoothly on the point, see \cite[Theorem 9]{Sasaki76}). Hence, $U_{(x,v)}$ is a smooth subset of $S\bH^2$, and for any $\eps>0$ small enough, $U_{(x,v),\eps}$ is diffeomorphic to $\exp_{\mathrm{Sas},(x,v)}\left(\lbrace z\in P_{(x,v)} \mid \lVert z \rVert <\eps\rbrace\right)$, which is a smoothly embedded disk.
\end{proof}

\subsubsection{A special parametrization of a neighborhood of $SN$}

We begin by fixing a lift $\wt N$ of $N$ inside $\wt M = \bH^n$. Note that since $N$ is a totally geodesic hypersurface in $M$, then, $\wt N$ is a totally geodesic hyperbolic space $\bH^{n-1}$ inside $\bH^n$.
For any $(x,v) \in SN$, we denote by $(\wt x, v)$ a lift in $S\wt N$.

Because $N$ is two-sided we can choose $u$ a vector field along $N$ which is normal to $N$. We denote by $\wt u$ the lift of $u$ to $\wt N$.

Now, for any $(\wt x,v) \in S\wt N$, we let $\bD_{(\wt x,v)} \subset \wt M$ be the totally geodesic hyperbolic $2$-plane which contains $\wt x$ and is tangent to $v$ and $\wt u (x)$. Then, we define $\wt U_{(\wt x,v)} \subset S \bD_{(\wt x,v)}$ as in Definition~\ref{def_special_subset}, and accordingly, for any $\eps$, we define $\wt U_{(\wt x,v),\eps}\subset \wt U_{(\wt x,v)} $ as in Proposition~\ref{properties_of_special_set}.

Because our definition of $\wt U_{(\wt x,v)}$ relies on the hyperbolic metric only and the covering $\pi_M \colon \wt M \rightarrow M$ is a local isometry, the following lemma is immediate.
\begin{lemma}\label{lem_invariance_by_fundamental_group}
 For all $\eps>0$, the disks $\wt U_{(\wt x,v),\eps}$ are $\pi_1(M)$-equivariant, \ie for any $\gamma \in \pi_1(M)$ which leaves $\wt N\subset \wt M$ invariant  we have
 \[
  \gamma \wt U_{(\wt x,v),\eps} = \wt U_{\gamma \cdot (\wt x,v), \eps}.
 \]
Further, for all $(x,v)\in SN$, the set 
\[
 U_{(x,v),\eps} = \pi_M \left( \wt U_{(\wt x,v),\eps} \right),
\]
 is well-defined.
\end{lemma}

Another easy property of the sets $\wt U_{(\wt x,v),\eps}$ is the following
\begin{lemma} \label{lem_empty_intersection}
 If $(\wt x,v) \neq (\wt x',v')  \in S\wt N$, then, for any $\eps>0$ small enough, 
 \[
\wt U_{(\wt x,v),\eps} \cap \wt U_{(\wt x',v'),\eps} = \varnothing.
 \]
\end{lemma}

\begin{proof}
The claim is immediate if the hyperbolic planes $\bD_{(\wt x,v)}$ and $\bD_{(\wt x',v')}$ are disjoint. So we can assume that they intersect. Then either they are equal or they intersect along a geodesic.

We first assume that $\bD_{(\wt x,v)}$ and $\bD_{(\wt x',v')}$ are distinct, and call $\gamma$ the geodesic given by their intersection. Given the definition of these discs, the intersection needs to be orthogonal to $\wt N$, i.e., along a geodesic $k_t(s)$. Hence, $ S\bD_{(\wt x,v)}\cap S\bD_{(\wt x',v')} = \{(k_t(s), \dot{k}_t(s))\}$, and $d_{\mathrm{Sas}}( (\wt x,v), (k_t(s), \dot{k}_t(s))) \geq \pi/2$ (where we choose the Sasaki metric such that the length of each great circle in the fiber is $2\pi$). Therefore we get the conclusion for all $\eps<\pi/2$.

So we are left to consider the case when $\bD_{(\wt x,v)} = \bD_{(\wt x',v')}$, that is if $(\wt x,v)$ and $(\wt x',v')$ define the same geodesic of $\wt N$ (with orientation potentially reversed). By the definition of $\wt U_{(\wt x,v),\eps}$, it is again clear that $\wt U_{(\wt x,v),\eps} \cap \wt U_{(\wt x',v'),\eps} = \varnothing$.
\end{proof}

\begin{definition}
 For any $\eps$, we define the sets $\cV_{\eps} \subset SM$ by
\[
 \cV_{\eps} = \bigcup_{(x,v) \in SN}  U_{(x,v),\eps}
\]
\end{definition}

We denote by $g^t\colon SM \rightarrow SM$ the geodesic flow.

\begin{proposition}
\label{prop_coordinates}
There exists $\eps_0>0$ such that, for all $\eps<\eps_0$, we have the following:
 \begin{enumerate}
  \item \label{tubular_neighborhood} The sets $\cV_{\eps}$ are neighborhoods of $SN$ and codimension $0$ embedded submanifolds. In particular, we have
 \[
  \cV_{\eps}= \lbrace (y,w)\in SM \mid d_{\mathrm{Sas}}((y,w), SN) \le \eps \rbrace \simeq SN \times \bD^2
 \]
  \item \label{flow_invariance_on_M} For any $t \in \R$, there exists $\eps'>0$, such that 
  \begin{align*}
   g^t\left( \cV_{\eps'}\right) &\subset \cV_{\eps} \\
   g^t\left( U_{(x,v),\eps'}\right) &\subset  U_{g^t(x,v),\eps}
  \end{align*}
  \item \label{flip_invariance_on_M} The sets $\cV_{\eps}$ are invariant under the flip map $\sigma \colon SM \rightarrow SM$, i.e., 
  \[
   \sigma \left(\cV_{\eps}\right) = \cV_{\eps},
  \]
where $\sigma (y,w) = (y, -w)$.
 \item \label{one_dimensional_foliations} For all $(x,v) \in SN$, the sets $U_{(x,v),\eps}$ are transverse to the geodesic flow. Moreover, the weak stable and weak unstable foliations of the geodesic flow $g^t$ of $M$ restrict to two transverse $1$-dimensional foliations of $U_{(x,v),\eps}$
  \item \label{linearization_of_flot} For all $(x,v) \in SN$, there exist coordinates $(s,u)$ on the disc $U_{(x,v),\eps}$ around $(x,v)$ such that:
\begin{enumerate}
 \item For all $z \in \cV_{\eps}$ and $t\in \R$ such that $g^t(z) \in \cV_{\eps}$, if $z = (s,u; x,v)$, then 
\[
 g^t(z) = \left(e^{-t}s , e^{t}u; g^t(x,v) \right).
\]
 \item For any $(x,v) \in SN$, and $u_0$ constant, we have $\left\{ (s,u;x,v) \in \cV_{\eps} \mid u = u_0 \right\} \subset \cF^{u}(x,v)$.
 \item For any $(x,v) \in SN$, and $s_0$ constant, we have $\left\{ (s,u;x,v) \in \cV_{\eps} \mid s = s_0 \right\} \subset \cF^{s}(x,v)$.
 \item For any $(x,v) \in SN$, we have
 \[
\left\{ (0,u;x,v) \in \cV_{\eps} \right\} \times \cF^{uu}_N(x,v)   = \cF^{uu}(x,v) \cap \cV_{\eps}, 
 \]
where $\cF^{uu}_N(x,v)$ is the restriction of the strong unstable leaf through $(x,v)$ to $SN$.
 \item For any $(x,v) \in SN$, we have
 \[
\left\{ (s,0;x,v) \in \cV_{\eps} \right\} \times \cF^{ss}_N(x,v)   = \cF^{ss}(x,v) \cap \cV_{\eps}, 
 \]
where $\cF^{ss}_N(x,v)$ is the restriction of the strong stable leaf through $(x,v)$ to $SN$.
\end{enumerate}
 \end{enumerate}
\end{proposition}

\begin{proof}
 Items \ref{flow_invariance_on_M} and \ref{flip_invariance_on_M} are direct consequences of items \ref{flow_invariance} and \ref{flip_invariance} of Proposition \ref{properties_of_special_set} together with the equivariance with respect to the fundamental group given by Lemma~\ref{lem_invariance_by_fundamental_group}.
  
 We will now prove item \ref{tubular_neighborhood}. Let $(y,w)\in SM$ be a point in a tubular neighborhood (of radius, say, $\eta$ for the Sasaki metric) of $SN$. Fix a lift $S\wt N$ of $N$ and $(\wt y,w)$ lift of $(y,w)$ in the tubular neighborhood of $S\wt N$.
 Let $\wt z \in \wt N$ be the orthogonal projection of $\wt y$ onto $\wt N$. Let $\bD_{(z,v_z)}$ be a hyperbolic plane in $\wt M \simeq \bH^3$ that is normal to $\wt N$, contains $\wt y$ and tangent to $w$ at $\wt y$. The plane $\bD_{(z,v_z)}$ is unique except if the hyperbolic geodesic through $\wt y$ in the direction $w$ is normal to $\wt N$ (equivalently, contains $z$). If this is the case, then it would imply that $d_{\mathrm{Sas}}((\wt y,w),S\wt N)\geq\pi/2$. So, if we assume that $\eta <\pi/4$, then the plane $\bD_{(z,v_z)}$ is uniquely determined. 
 From here, it is easy to see that there exists a unique $(\wt x,v)\in S\wt N$, on the geodesic determined by $(z,v_z)$, such that $(\wt y, w)\in U_{(\wt x,v)}$. Indeed, if we call $c_{v_z}(t)$ the geodesic in $\wt N$ determined by $(z,v_z)$, then one can notice that 
 \[
  \cup_{t\in \mathbb{R}} U_{(c_{v_z}(t), \dot{c_{v_z}}(t))} \supset S\bD_{(z,v_z)},
 \]
which implies the existence of $(\wt x,v)$, and its uniqueness is also immediate. Hence, we proved that for any $\eta$ small enough, the tubular neighborhood of radius $\eta$ around $SN$ is contained in $\cV_{\eps}$ for all $\eps \ge \eta$. Now, by definition of $U_{(x,v),\eps}$, we also have that $\cV_{\eps}$ is contained in the tubular neighborhood around $SN$ of any radius $\eta \ge \eps$. Therefore, 
\[
  \cV_{\eps}= \lbrace (y,w)\in SM \mid d_{\mathrm{Sas}}((y,w), SN) \le \eps \rbrace.
\]

 To prove item \ref{one_dimensional_foliations}, we go back to the universal cover. By definition of $\wt U_{(\wt x,v)}$, there exists a $1$-dimensional submanifold $\wt \cF^{uu}_1(\wt x,v) \subset \wt \cF^{uu}(\wt x,v)$ such that $\wt \cF^{uu}_1(\wt x,v) \subset \wt U_{(\wt x,v)}$. Similarly, there exists a $1$-dimensional submanifold $\wt \cF^{ss}_1(\wt x,v) \subset \wt \cF^{ss}(\wt x,v)$ such that $\wt \cF^{ss}_1(\wt x,v) \subset \wt U_{(\wt x,v)}$.
 Now, since $\wt U_{(\wt x,v)} \subset S\bD_{(\wt x,v)}$ and for $\eps>0$ small enough ($\eps<\pi$), $\wt g^t(\wt x,-v) \cap \wt U_{(\wt x,v), \eps} =\emptyset$, we have for any $(y,w) \in \wt U_{(\wt x,v)}$, 
 \[
 \wt \cF^{s}(y,w) \cap  \wt \cF^{uu}_1(\wt x,v) \neq \emptyset.
 \]
Indeed, in the unit tangent bundle of the hyperbolic plane, every weak stable leaf $L^s$ intersect all the strong unstable leaves, except for the ones on the weak unstable leaf $L^u$ obtained from $L^s$ by rotation by $\pi$.
Hence, we deduce that $\wt \cF^{s}$ restrict to a $1$-dimensional foliation on $\wt U_{(\wt x,v), \eps}$, and similarly for $\wt \cF^{u}$. Projecting down to $M$ gives item \ref{one_dimensional_foliations}.

Thanks to item \ref{one_dimensional_foliations}, we can put coordinates on each $U_{(x,v),\eps}$ in the following way.
We parametrize the submanifolds in the strong leaves $\cF^{uu}_1( x,v)$ and $\cF^{ss}_1( x,v)$ by the arc-length of their projections to $M$. Hence, we can set
\begin{align*}
 \cF^{uu}(x,v) \cap U_{(x,v),\eps} &= \{ (0,u; x,v) \mid |u|< \eps \} \\
 \cF^{ss}(x,v) \cap U_{(x,v),\eps} &= \{ (s,0; x,v) \mid |s|< \eps \},
\end{align*}
where $u$ is the arc-length along $\pi_M\left(\cF^{uu}_1( x,v)\right)$ and $s$ the arc-length along $\pi_M\left(\cF^{ss}_1( x,v)\right)$ (which also happen to correspond to the arc-length along $\cF^{uu}_1( x,v)$ and $\cF^{ss}_1( x,v)$ for the Sasaki metric).

Now, for any $(y,w) \in U_{(x,v),\eps}$, there exists $u$, and $s$ such that
\begin{align*}
 \cF^{s}(y,w) \cap  \cF^{uu}_1(\wt x,v) &= (0,u; x,v)  \text{ and} \\
 \cF^{u}(y,w) \cap  \cF^{ss}_1(\wt x,v) &= (s,0; x,v), 
\end{align*}
so, we set $(y,w) = (s,u; x,v)$.

Then items \ref{linearization_of_flot} b) and c) hold by our definition. Item \ref{linearization_of_flot} a) follows directly from the flow invariance (item \ref{flow_invariance_on_M}) and the fact that, since our metric is hyperbolic, the distance along the strong stable and unstable leaves are contracted/expanded exactly by the factor $e^t$. 
\end{proof}

\subsection{The DA flow} \label{sec_DA_flow}

In this section we deform the geodesic flow $g^t\colon SM\to SM$ in the neighborhood $\cV_\eps$ of $SN$ to obtain the {\it DA flow} $\phi^t$ such that $SN$ is a repeller for $\phi^t$ and the attractor of $\phi^t$ is hyperbolic with $(n-1, 1, n-1)$ splitting. This is standard and well-known to the experts. However we were not able to locate any reference with a rigorously constructed DA flow, so we provide a reasonable amount of details. 

We pick a small $\eps>0$ --- how small to be specified later. Then Proposition~\ref{prop_coordinates} gives the $(s,u,v)$ coordinates on $\cV_\eps$ and we will heavily rely on these coordinates.

Let $\alpha\colon\R\to[0,\sqrt 2]$ be a smooth bump function such that
$$
\alpha(x)=
\begin{cases}
\sqrt 2, \, x=0\\
0, \, |x|\ge 1
\end{cases}
$$
We also assume that $\alpha$ is even and that it is decreasing on $[0,1]$.
Pick $\delta\ll\eps$ --- how much smaller to be specified later --- and define a 2-dimensional bump function
$$
\beta(s,u)=\alpha\Big(\frac s\delta \Big)\alpha\Big(\frac u\delta \Big),
$$
which we now use to define the flow $\phi^t\colon SM\to SM$ as follows. Let $\cU_\delta\subset\cV_\eps$ be given by 
$\cU_\delta=\{(s,u,v): |s|\le \delta, |u|\le \delta\}$ and let 
$$
\dot\phi^t=\dot g^t
$$
outside $\cU_\delta$. When $|s|\le\delta$ and $|u|\le \delta$ we let
$\dot \phi^t=(\dot s^t, \dot u^t, \dot g_N^t)$, where
\begin{equation}
\label{eq_su}
\begin{cases}
\dot s^t= (\beta(s,u)-1)s\\
\dot u^t =u
\end{cases}
\end{equation}
and $g_N$ is the geodesic flow on $SN$.
Recall that by  Proposition~\ref{prop_coordinates}, inside $\cV_\eps$, the geodesic flow is a product of the saddle flow and the geodesic flow on $SN$. Therefore one can easily check that the above definitions yields a smooth flow $\phi^t$. It is also clear that $SN$ is a $\phi^t$-invariant repeller. (Note that $\phi^t$ depends on $\delta$ but we have dropped this dependence from our notation in order to keep notation lighter.)

\begin{prop}
\label{prop_DA} 
There exist an $\eps>0$ and a $\delta>0$ such that the flow $\phi^t\colon SM\to SM$ is a DA flow. That is,
\begin{enumerate}
\item $SN\subset SM$ is a hyperbolic repeller;
\item there is an open neighborhood of the repeller $\cV_\kappa\subset\cU_\delta$ such that the set $\Lambda=SM\backslash \cup_{t\ge 0}\phi^t(\cV_{\kappa})$ is a hyperbolic attractor with $(n-1, 1, n-1)$ splitting;
\item the weak stable manifolds of $\phi^t|_\Lambda$ are submanifolds of the weak stable manifolds for $g^t$.
\end{enumerate}
\end{prop}

\begin{proof}
The fact that $SN$ is a $\phi^t$-invariant repeller is immediate from the definition of $\phi^t$. We pick a small neighborhood $\cV_\kappa$ of $SN$ such that $\dot\phi^t$ is transverse to $\partial\cV_\kappa$. Then the set  $\Lambda=SM\backslash \cup_{t\ge 0}\phi^t(\cV_{\kappa})$ is a closed $\phi^t$-invariant attractor. To prove hyperbolicity of $\Lambda$ we will follow the usual route of constructing the stable and unstable cones. Note that $\phi^t$ is constructed by ``pushing away" along the horizontal curves $\{s=\mbox{const}\}$ inside $\cV_\eps$. According to Proposition~\ref{prop_coordinates} these horizontal curves subfoliate the weak stable foliation of $g^t$. Therefore, the weak stable foliation of $g^t$ is invariant under $\varphi^t$  and we can look for stable cones $\cC^{ss}$  for $\phi^t$ inside of the weak stable distribution $E^s$, which will be easy. As a byproduct we will automatically obtain that weak stable leaves for $\Lambda$ are contained (some are not complete) in the weak stable leaves of $g^t$ as posited in item 3 of the proposition. Constructing the unstable cones $\cC^{uu}$ is more challenging and we will need to make the parameter $\delta$ sufficiently small to assure invariance and expansion of the unstable cones. We proceed to explain the construction of the cones.\\
{\it The domains.} Consider the set $\cE\subset [-\delta, \delta]\times[-\delta, \delta]$ given by the condition $\beta(s,u)>1$. Then, according to the definition of $\beta$, this is a symmetric oval-shaped domain. Note that according to~(\ref{eq_su}) we have $\dot s=0$ on the boundary $\partial\cE$ and, hence, the flow $(s^t,u^t)$ is a vertical outward flow on $\partial \cE$ except at two fixed points $(-\bar s,0)$ and $(\bar s, 0)$ which are indicated on Figure~\ref{fig_DA}. Here, $\bar s$ and $-\bar s$ are the solutions to $\beta(s,0)=1$. Hence, with the exception of the points $(-\bar s, 0)$ $(\bar s, 0)$ the flow $(s^t, u^t)$ is an outward flow on $\partial\cE$. We let
$$
\hat \cE=(s^{t_0}, u^{t_0})(\cE)\supset \cE,
$$
for some small $t_0>0$. From now on consider $\hat\cE$ instead of $\cE$ because it has the following property
\begin{equation}
\label{eq_bound}
\forall (s,u)\notin\hat\cE \;\;\;\; \beta(s,u)<1 \,\,\, \mbox{unless} \,\,\,\, (s,u)=(\pm \bar s, 0)
\end{equation}

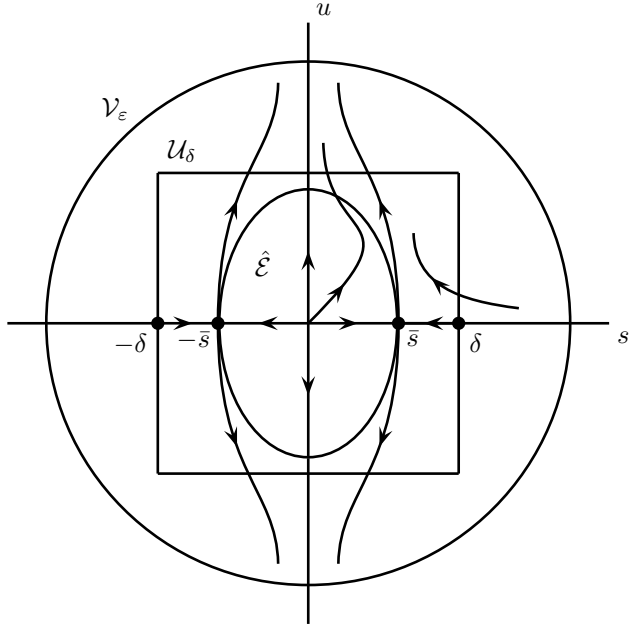
\begin{figure}
\centering
\begin{pspicture}(-4,-4)(4,4)
  \pscircle[linewidth=1pt]{3.5}

  \psline[linewidth=1pt](2,-2)(2,2)
\psline[linewidth=1pt](-2,-2)(-2,2)
\psline[linewidth=1pt](-2,-2)(2,-2)
\psline[linewidth=1pt](-2,2)(2,2)

\psline[linewidth=1pt](0,-4)(0,-1.8)
\psline[linewidth=1pt,ArrowInside=->,arrowscale=1.5](0,0)(0,-1.8)
\psline[linewidth=1pt](0,4)(0,1.8)
\psline[linewidth=1pt,ArrowInside=->,arrowscale=1.5](0,0)(0,1.8)

\psline[linewidth=1pt](-4,0)(-2,0)
\psline[linewidth=1pt,ArrowInside=->,arrowscale=1.5](0,0)(-1.2,0)
\psline[linewidth=1pt,ArrowInside=->,arrowscale=1.5](-2,0)(-1.2,0)
\psline[linewidth=1pt](4,0)(1.2,0)
\psline[linewidth=1pt,ArrowInside=->,arrowscale=1.5](0,0)(1.2,0)
\psline[linewidth=1pt,ArrowInside=->,arrowscale=1.5](2,0)(1.2,0)

\psellipse[linewidth=1pt](0,0)(1.2,1.8)

\psbezier[linewidth=1pt,ArrowInside=->,arrowscale=1.5](1.2,0)(1.2,2)(0.42,2.2)(0.4,3.2)
\psbezier[linewidth=1pt,ArrowInside=->,arrowscale=1.5](1.2,0)(1.2,-2)(0.42,-2.2)(0.4,-3.2)

\psbezier[linewidth=1pt,ArrowInside=->,arrowscale=1.5](-1.2,0)(-1.2,2)(-0.42,2.2)(-0.4,3.2)
\psbezier[linewidth=1pt,ArrowInside=->,arrowscale=1.5](-1.2,0)(-1.2,-2)(-0.42,-2.2)(-0.4,-3.2)

\psbezier[linewidth=1pt,ArrowInside=->,arrowscale=1.5,ArrowInsidePos=0.25](0,0)(1.5,1.5)(0.24,1)(0.2,2.4)

\psbezier[linewidth=1pt,ArrowInside=->,arrowscale=1.5,ArrowInsidePos=0.65](2.8,0.2)(2,0.3)(1.42,0.5)(1.4,1.2)

\psdot[dotsize=5pt](-1.2,0)
\psdot[dotsize=5pt](1.2,0)
\psdot[dotsize=5pt](-2,0)
\psdot[dotsize=5pt](2,0)
\uput{0.2}[-135](-2,0){$-\delta$}
\uput{0.2}[-45](2,0){$\delta$}
\uput{0.15}[-135](-1.2,0){$-\bar s$}
\uput{0.15}[-45](1.2,0){$\bar s$}
\rput(-0.6,0.8){$\hat{\mathcal{E}}$}
\uput{0.2}[45](-2,2){$\mathcal{U}_{\delta}$}
\put(-2.75,2.75){$\mathcal{V}_{\varepsilon}$}
\uput{0.15}[-45](4,0){$s$}
\uput{0.15}[45](0,4){$u$}
\end{pspicture}
      \caption{The DA flow in the neighborhood $\mathcal{V}_{\varepsilon}$}
      \label{fig_DA}
\end{figure}

According to our definition of $\hat\cE$ we have $(\hat\cE\times SN)\cap\Lambda=\varnothing$ and $(\partial\hat\cE\times SN)\cap\Lambda=\{(-\bar s,0), (\bar s,0)\}$. Hence, to show that $\Lambda$ is a hyperbolic set it is sufficient to construct invariant stable/unstable cones on the set $SM\backslash (\hat\cE\times SN)$, which contains $\Lambda$. Further, we subdivide $SM\backslash (\hat\cE\times SN)$ into three subdomains.
\begin{enumerate}
\item $D_G=SM\backslash\cV_\eps$ --- the domain which is far from the support of the DA deformation (and we will be able to use the cones of the geodesic flow there);
\item $D_P=\cU_\delta\backslash (\hat \cE\times SN)$  --- the domain where the deformation took place with $\hat\cE\times SN$ excised (and the cones have to be carefully constructed here);
\item $D_T=\cV_\eps\backslash \cU_\delta$ --- the transition region (which is needed to interpolate the cones between $D_G$ and $D_P$).
\end{enumerate}
{\it The stable cones.} Recall that $\phi^t$ preserves the weak stable distribution $E^s$ of $g^t$. Denote by $E^{ss}$ the $(n-1)$-dimensional strong stable distribution of $g^t$ and let $\cC^{ss}$ be the cone fields about $E^{ss}$ inside $E^s$ of a fixed small aperture (= the angle of the cone), say $\pi/10$. We claim that outside of $\hat \cE\times SN$ the cone family $\cC^{ss}$ is {\it invariant}, \ie for all $t\ge 0$
$$
\phi^{-t}(\cC^{ss})\subset \cC^{ss}
$$
and {\it contracting}, \ie there exists $\mu>1$ such that for all $t\ge 0$ and $v\in\cC^{ss}$
$$
\|D\phi^{-t}v\|\ge \mu^t\|v\|.
$$
Within domains $D_G$ and $D_T$ these properties hold just because $\phi^t=g^t$. Within the domain $D_P$, recall
\begin{align*}
& g^t(s_0, u_0, v)=(e^{-t}s_0, e^tu_0, g^t_N(v)),\\
& \phi^t(s_0,u_0, v)=(s^t(s_0, u_0), u^t(s_0, u_0), g_N^t(v))
\end{align*}
and $u^t(s_0, u_0)=e^tu_0$. By Proposition~\ref{prop_coordinates}, within $\cV_\eps$ the leaves of the weak stable foliation $\cF^s$ have the product form $\{(s,u): u = const\}\times \cF^s_N$. Further, again by Proposition~\ref{prop_coordinates},  the (strong) stable leaves through $(0, 0, v)$, $v\in SN$, are given by $\{(s,0)\}\times \cF^{ss}_N$. It follows that for sufficiently small $\delta$ the cone field $\cC^{ss}$ on $\cU_\delta$ contains the sum distribution $\frac{\partial}{\partial s}\oplus E_{SN}^{ss}$. Hence, to show that $\cC^{ss}$ is invariant and contracting it suffices to show that $s^t$ is an exponential contraction. We calculate
$$
\frac{d(\log(\dot s^t))}{dt}=\frac{d\dot s^t}{ds}=(\beta-1)+\frac{\partial\beta}{\partial s}s
$$
Now recall that outside $\hat\cE$ we have $\beta-1\le 0$. Also $\frac{\partial\beta}{\partial s}s\le 0$. Moreover, 
$\frac{\partial\beta}{\partial s}( \bar s, 0) \bar s< 0$
by our choice of $\beta$. Hence by combining with~(\ref{eq_bound}) we indeed have
\begin{equation}
\label{eq_sigma}
(\beta-1)+\frac{\partial\beta}{\partial s}s<-\xi<0,
\end{equation}
for some $\xi>0$.\\

\noindent {\it The unstable cones: the scheme.} The unstable cones have to be constructed, of course, inside of the full tangent space. Inside the domain $D_G$ we define $\cC^{uu}$ as the cone field about $E^{uu}$ of aperture $\pi/10$. Within the domain $D_P$ the flow $\phi^t$ has the product form and, hence, it is convenient to look for $\cC^{uu}$ in the product form
\begin{equation}
\label{eq_product}
\cC^{uu}=\cC^{uu}_{l-r} \times \cC^{uu}_{SN},
\end{equation}
where $\cC^{uu}_{l-r}$ is a 2-dimensional cone family to be carefully defined on $[-\delta, \delta]\times[-\delta, \delta]\backslash \hat\cE$ and $\cC^{uu}_{SN}$ is the cone field in $T(SN)$ about $E^{uu}_{SN}$ of aperture $\pi/10$.

Note that by our choice of domains there exists an upper bound on time which an orbit of $\phi^t$ can spend inside of $D_T=\cV_\eps\backslash\cU_\delta$. (This upper bound, of course, depends on $\delta$.) Therefore we do not need to explicitly define the cones inside $D_T$. However we will need to check invariance of the cone field within $D_G$ and $D_P$ and also when an orbit travels from $D_G$ to $D_P$ or vice versa. Also we will check eventual expansion of vectors in the cones. Then this will be sufficient to conclude hyperbolicity of $\Lambda$ from the standard cone criterion (see, for instance, \cite{KH}).

\noindent {\it The unstable cones: invariance.} Clearly $\cC^{uu}$ is invariant in $D_G$. Because $\phi^t=(s^t, u^t, g_N^t)$ in $D_P$, according to our definition~(\ref{eq_product}), to check invariance of $\cC^{uu}$ within $D_P$ we only need to make sure that $\cC^{uu}_{l-r}$ is invariant under $(s^t, u^t)$. 

We define $\cC^{uu}_{l-r}$ using $(s,u)$ coordinates as a ``constant" cone family. Namely, given a point $(s_0, u_0)\in [-\delta, \delta]\times[-\delta, \delta]\backslash \hat\cE$ we set $\cC^{uu}_{l-r}$ to be the cone of vectors based at $(s_0, u_0)$ which lie between
$v_l=
\left(
\begin{smallmatrix}
l\\
1
\end{smallmatrix}
\right)
$
and
$v_r=
\left(
\begin{smallmatrix}
r\\
1
\end{smallmatrix}
\right)
$
(the cone contains the vertical vector
$
\left(
\begin{smallmatrix}
0\\
1
\end{smallmatrix}
\right)
$
).
Here $l<0<r$ are constants which are to be determined. To establish invariance of $\cC^{uu}_{l-r}$ under $(s^t, u^t)$ we need to check that $v_l$ rotates clockwise under $D\phi^t$ and $v_r$ rotates counter-clockwise. This is a local property. To check it we linearize $(s^t, u^t)$. The Jacobian at a point $(s_0, u_0)$ is given by
\begin{multline*}
J=J(s_0, u_0)=
\begin{pmatrix}
\frac{\partial\dot s^t}{\partial s} & \frac{\partial\dot s^t}{\partial u}  \\
\frac{\partial\dot u^t}{\partial s} & \frac{\partial\dot u^t}{\partial u}  
\end{pmatrix}
=\\
\begin{pmatrix}
\beta(s_0, u_0)-1+ \frac{\partial\beta}{\partial s}(s_0, u_0)s_0& \frac{\partial\beta}{\partial u}(s_0, u_0)s_0 \\
0 & 1
\end{pmatrix}
\stackrel{\mathrm{def}}{=}
\begin{pmatrix}
-\sigma(s_0, u_0)& -c(s_0, u_0) \\
0 & 1
\end{pmatrix}
\end{multline*}
By~(\ref{eq_sigma}) we have $\sigma>0$. It is also clear that $c\ge 0$. We will also need the following upper bound on $c$.
\begin{equation}
\label{eq_c}
c(s_0, u_0)=-\alpha\left(\frac{s_0}\delta\right)\frac1\delta\alpha'\left(\frac{u_0}\delta\right)s_0< \sqrt2\frac1\delta\max|\alpha'|\delta
=\sqrt2\max|\alpha'|\stackrel{\mathrm{def}}{=}\bar c.
\end{equation}
Note that this upper bound is independent of $\delta$.

Now, the invariance of $\cC^{uu}_{l-r}$ can be expressed infinitesimally as follows
$$
\langle J
\left(
\begin{smallmatrix}
r\\
1
\end{smallmatrix}
\right), 
\left(
\begin{smallmatrix}
-1\\
r
\end{smallmatrix}
\right)\rangle>0
\,\,\,\,\,\,\mbox{and}\,\,\,\,\,\,
\langle J
\left(
\begin{smallmatrix}
l\\
1
\end{smallmatrix}
\right), 
\left(
\begin{smallmatrix}
1\\
-l
\end{smallmatrix}
\right)\rangle>0.
$$
The first inequality simplifies as
$$
\sigma r+c+r>0,
$$
which clearly holds for any $r>0$. The second inequality yields the following condition on $l$
$$
l<\frac{-c}{\sigma+1}.
$$
Hence we define $\cC^{uu}_{l-r}$ (and hence  $\cC^{uu}$) by setting $l=-\bar c$ and $r=\bar c$. Then, using~(\ref{eq_c}) we have
$$
l< -c<\frac{-c}{\sigma+1},
$$
which implies that $\cC^{uu}_{l-r}$ is invariant in $D_P$ by the above discussion.

To complete the proof of invariance of $\cC^{uu}$ we need to show that if an orbit of $\phi^t$ travels from $D_G$ to $D_P$ then the cone  at the starting point of the orbit from the cone field $\cC^{uu}|_{D_G}$ maps inside $\cC^{uu}|_{D_P}$.  (We also need the symmetric property for orbits travelling from $D_P$ to $D_G$. Its proof is completely analogous, so we omit it.)

We extend $\cC^{uu}$ to the boundary $\partial D_G\simeq \partial \cV_\eps$. To see this invariance property recall that the distribution $E^{uu}$ is the axis of $\cC^{uu}|_{D_G\cup \partial D_G}$.  Further, for small $\e$, $E^{uu}$ is very close (but does not coincide) to the sum distribution $\frac\partial{\partial u}\oplus E^{uu}_N$ inside $\cV_\eps$ and the angle between these distributions goes to zero as the base point approaches $SN\subset \cV_\eps$. It follows that, provided that $\eps$ is sufficiently small, the cone field
$\cC^{uu}|_{\partial D_G}$ is contained inside the cone field $\hat \cC^{uu}|_{\partial D_G}$ which is the cone field about $\frac\partial{\partial u}\oplus E^{uu}_N$ and has aperture $\pi/9$. (Recall that aperture of $\cC^{uu}|_{D_G}$ is $\pi/10$.)

Now both $\cC^{uu}|_{D_P}$ and $\hat \cC^{uu}|_{\partial D_G}$ are cone fields centered at the same distribution $\frac\partial{\partial u}\oplus E^{uu}_N$. (Note that according to our definitions $\cC^{uu}|_{D_P}$ is a ``product cone field" while $\hat \cC^{uu}|_{\partial D_G}$ is a ``round cone field", but this is a minor technicality which does not cause any trouble.) Apertures of both of these cone fields are fixed and do not depend on $\delta$. Further, the distribution $\frac\partial{\partial u}\oplus E^{uu}_N$ stays $D\phi^t$-invariant as an orbit crosses $D_T$ on its way from $D_G$ to $D_P$. And from the product form of the flow, the distribution $\frac\partial{\partial u}\oplus E^{uu}_N$ is exponentially expanding within $D_T$ while the complementary distribution $\frac\partial{\partial s}\oplus E^{ss}_N$ is exponentially contracting. It remains to notice that the time it takes to reach $D_P\subset\cU_\delta$ from $D_G$ increases as $\delta\to 0$. Hence, by choosing sufficiently small $\delta$, we ensure that cones $\hat \cC^{uu}|_{\partial D_G}$ map inside the cones $\cC^{uu}|_{D_P}$ for orbits traveling from $D_G$ to $D_P$. Hence the cones $\cC^{uu}|_{\partial D_G}$ map inside $\cC^{uu}|_{D_P}$ as well.

\begin{remark}
The distribution $\frac\partial{\partial u}\oplus E^{uu}_N$ is {\it not } globally invariant under $D\phi^t$, but it stays invariant for finite orbit segments in $\cV_\eps$.
\end{remark}

\noindent {\it The unstable cones: expansion.} Because $\phi^t|_{D_T\cup D_G}=g^t|_{D_T\cup D_G}$ the cone field is infinitesimally expanding and, moreover, the vectors from $\cC^{uu}|_{D_G}$ which travel to $D_T$ keep expanding while in $D_T$. Therefore we only need to establish exponential expansion property for $\cC^{uu}|_{D_P}$. We will show that vectors in $\cC^{uu}|_{D_P}$ are {\it eventually expanding} \ie
$$
\exists C>0: \,\forall t>0\,\, \forall v\in \cC^{uu}|_{D_P}\,\, \|D\phi^t v\|\ge C\mu^t\|v\|
$$
where $\mu>1$.
Eventual expansion of vectors in the cones is a weaker property, but still sufficient for hyperbolicity. 

Recall that  $\cC^{uu}|_{D_P}$ is a product~(\ref{eq_product}) and hence, obviously, we only need to show that $\cC^{uu}_{l-r}$ is eventually expanding under $(s^t, u^t)$. Recall that 
$$(s^t, u^t)(s_0, u_0)=(s^t(s_0, u_0), e^tu_0)$$
 and, hence,
$$
D(s^t, u^t)=
\begin{pmatrix}
* & * \\
0 & e^t
\end{pmatrix}
$$
For a vector 
$v=
\left(
\begin{smallmatrix}
a\\
1
\end{smallmatrix}
\right)\in \cC^{uu}_{l-r}$ (\ie $|a|\le \bar c$) we calculate
$$
D(s^t, u^t)v=
\left(
\begin{smallmatrix}
\hat a\\
e^t
\end{smallmatrix}
\right).
$$
And now we verify eventual expansion as follows
$$
\frac{\|D(s^t, u^t)v\|^2}{\|v\|^2}=\frac{\hat a^2+e^{2t}}{a^2+1}\ge \frac{e^{2t}}{\bar c^2+1}.
$$
Hence we have established hyperbolicity of $\Lambda$ by verifying the assumptions of the cone criterion.
\end{proof}

\subsection{Anosov flow on the twisted double}

We are now ready to give the
\begin{proof}[Proof of Theorem~\ref{thm2}]
We follow closely the argument of Franks-Williams~\cite{FW}. Recall that, according to the discussion in Subsection \ref{subsec_total_geod_hypersurface}, for any $n\ge 2$, we can find an $n$-dimensional hyperbolic manifold $M$ with a totally geodesic hypersurface $N\subset M$.
Then we can use the DA construction explained above and our starting point is the DA flow $\phi^t\colon SM\to SM$ given by Proposition~\ref{prop_DA}.
We  excise a small neighborhood of $SN$,
$$
\cV_\kappa=\{x\in SM: d_\mathrm{Sas}(x, SN)\le \kappa \},
$$
 out of $SM$. This gives a manifold with boundary $(W_1,\partial W_1)$ which we equip with the DA flow $\phi^t$. By Proposition~\ref{prop_DA} the flow $\phi^t$ is transverse to $\partial W_1$ provided that $\kappa>0$ is sufficiently small. Then we take a second copy of the same manifold which we denote by $(W_2,\partial W_2)$ and equip it with the flow $\hat\phi^t=\phi^{-t}$, that is, we reverse the time direction of the DA flow. Given a diffeomorphism $\omega\colon \partial W_1\to\partial W_2$ we can form a closed manifold $W$ by pasting together $W_1$ and $W_2$ along their boundaries using $\omega$
$$
W=W_1\cup_\omega W_2.
$$
The flows $\phi^t$ and $\hat \phi^t$ also paste together and yield a flow $\psi^t$. (In fact, for the flow $\psi^t$ to be well defined one also has to add a collar between $W_1$ and $W_2$ and extrapolate $\phi^t$ and $\hat\phi^t$ so that they are both orthogonal to the core of the collar, this point was thoroughly addressed in~\cite{FW}.)

According to Proposition~\ref{prop_DA}, the flow $\psi^t\colon W\to W$ is a flow with a hyperbolic attractor $\Lambda_1$ and a hyperbolic repeller $\Lambda_2$. The remaining points flow from $\Lambda_2$ to $\Lambda_1$ and hence are wandering. Moreover, by construction, both hyperbolic sets $\Lambda_1$ and $\Lambda_2$ have $(n-1,1,n-1)$ splitting. According to a criterion of Ma\~n\'e~\cite{Man} the flow $\psi^t$ is Anosov provided that we can additionally verify that the (weak) stable foliation of $\Lambda_1$ and the (weak) unstable foliation of $\Lambda_2$ intersect transversely. Intersections between stable leaves of $\Lambda_1$ and unstable leaves of $\Lambda_2$ occur along wandering orbits.

Again by construction, each wandering orbit intersects $\partial W_1$ transversely. Therefore, it is enough to check the transversality property at the points on $\partial W_1$, which we can identify with $\partial\cV_\kappa$.
By Proposition~\ref{prop_DA} (the $W_1$ component of) the stable manifold of a point in $\Lambda_1$  is embedded into the stable manifold of $g^t$ and (the $W_2$ component of) the unstable manifold of a point in $\Lambda_2$ is embedded into the unstable manifold of $g^{-t}$ (= stable manifold of $g^t$).
 It follows that the flow $\psi^t$ satisfies Ma\~n\'e's transversality condition provided that
 $\omega(\bar\fs)\pitchfork \bar\fs$, where $\bar \fs=\fs\cap \partial\cV_\kappa$.
 
 In order to construct such $\omega$ recall that by Proposition~\ref{prop_coordinates} the flip map $\sigma\colon SM\to SM$ given by $\sigma(v)=-v$ leaves $\cV_\kappa$ invariant.
 Clearly, the flip map $\sigma$ maps flow lines to flow lines, reversing the time, and it sends stable horocycles to unstable horocycles. Hence, letting $\omega=\sigma|_{\partial \cV_\kappa}$, we have
 $\omega(\bar\fs)=\bar\fu$, where $\bar \fu=\fu\cap \partial\cV_\kappa$. Clearly, $\bar \fs$ is transverse to $\bar \fu$ and, hence, the proof is complete.
 \end{proof}

We end this article by sketching an argument to obtain transitive examples in a fairly easy way. The idea is to take an arbitrarily large cover in order to garantee hyperbolicity. This argument was recently used in \cite{BGP} and similar argument were used before, see e.g. \cite{FJ78,FG12}. Note that passing to a finite cover changes the homotopy type of the manifold unless the manifold admits self covers (such as tori).
\begin{prop}\label{prop_transitive_examples}
 For any odd $n$, there exists a non-algebraic, transitive Anosov flow on a $n$-manifold.
\end{prop}

\begin{proof}[Sketch of proof]
 Let $M$ be a $d$-dimensional hyperbolic manifold which admits a non-separating totally geodesic hypersurface $N$. For any integer $k\ge 1$ one can construct a $k$-fold cover of $M$ in the following way. Cut $M$ along $N$ to obtain a manifold $M'$ with boundary $\partial M'= N_a\cup N_b$, two copies of $N$. Now take $k$ copies of $M'$, and glue them cyclically along their boundaries to obtain a (closed) manifold $M_k$. By taking $k$ large enough, we obtain a hyperbolic manifold $M_k$ which contains two totally geodesic hypersurfaces $N_1$ and $N_2$ which are as far from each other as we would like them to be.

 Now consider the geodesic flow $g^t$ on $SM_k$. One can perform a repelling DA construction for $g^t$ along $SN_1$ (as in Section~\ref{sec_DA_flow}), and an attracting DA along $SN_2$. Call $\varphi^t$ the flow obtained in this way. So $SN_1$ is a repeller of $\varphi^t$, $SN_2$ is an attractor, and there exists two neighborhoods $\cV_1$ and $\cV_2$ of $SN_1$ and $SN_2$ such that 
 $SM_k \smallsetminus \left(\cup_{t\ge 0} \varphi^t(\cV_1) \cup \cup_{t\ge 0} \varphi^{-t}(\cV_2)\right)$ is an invariant set for $\varphi^t$.
 To show that this set is in fact hyperbolic, one can redo the cone argument in the proof of Proposition \ref{prop_DA}, except that, now, one needs to consider certain modified cones (\ie the unstable cones in the proof of Proposition \ref{prop_DA}) for \emph{both} the stable and unstable cones. The fact that the distance between $SN_1$ and $SN_2$ can be assumed to be arbitrarily large, allows one not to check invariance explicitly with the cones going from a neighborhood of $SN_1$ to a neighborhood of $SN_2$, since enough hyperbolicity will be gained along the way.
 
 Now, as in the proof of Theorem \ref{thm2}, we can consider the flow $\varphi^t$ on the manifold $SM_k\smallsetminus\left(\cV_1\cup\cV_2\right)$ and glue $\partial\cV_1$ onto $\partial\cV_2$ using the flip map. Let $V_k$ be the manifold obtained by such gluing, it supports a flow $\bar{\varphi}^t$ obtained from $\varphi^t$ by gluing.
 
 Notice that, since we did two opposite DAs on the same manifold, we cannot preserve the stable/unstable foliations anymore. Hence the gluing using the flip map does not send stable to stable and unstable to unstable, as opposed to the non-transitive example above. However, the proof of Proposition \ref{prop_DA} shows that the new stable/unstable directions are in a cone of aperture at most $\pi/9$ from the stable/unstable directions of the original geodesic flow. Hence, the flip map will still glue the stable foliations transversely to the unstable ones.
 
However, just verifying transversality is not enough, as we cannot use Ma\~n\'e's characterization anymore to prove that $\bar{\varphi}^t$ is Anosov. So one has to go back to the cone criterion. The difficulty of proving that the cone criterion is verified for general flows build in such a way is to deal with cones along the orbits that will pass through the gluing repeatedly.
 Once again, taking a sufficiently large cover $M_k$, allows us to sidestep this difficulty. Since the distance between $N_1$ and $N_2$ in $M_k$ can be as large as we want, any orbit that goes from $\partial\cV_2$ to $\partial\cV_1$ will spend a very long time in $SM_k\smallsetminus\left(\cV_1\cup\cV_2\right)$, picking up enough hyperbolicity along the way to correct for whatever happened in the neighborhood of the gluing. Hence, for $k$ big enough, the flow $\bar{\varphi}^t$ on $V_k$ is Anosov.
 
 The fact that $\bar{\varphi}^t$ is not algebraic is easy to establish. For instance, one can notice that there exists pairs of distinct periodic orbits that are freely homotopic (and freely homotopic to the inverse of another pair of orbits). These are the periodic orbits contained in the two copies of the invariant submanifolds homeomorphic to $SN_1$ that are created by the DA construction. Whereas, the periodic orbits of $g^t$ that stayed away from the neighborhood of $SN_1$ and $SN_2$ are still periodic orbits of $\bar{\varphi}^t$ and are alone in their free homotopy class.
 
 Finally, we need to show that $\bar{\varphi}^t$ is transitive. To do that, we use the idea described in \cite{BBY} in the $3$-dimensional case.
 First, one follows the proof of item 6 of \cite[Proposition 7.1]{BBY} to obtain that the flow $\varphi^t$ is transitive on its invariant set $\Lambda=SM_k \smallsetminus \left(\cup_{t\ge 0} \varphi^t(\cV_1) \cup \cup_{t\ge 0} \varphi^{-t}(\cV_2)\right)$.
 Second, one needs to show that the local stable (resp. unstable) foliation of any point on the gluing hypersurface $\partial \cV_1 \simeq \partial \cV_2$ intersects the unstable (resp. stable) foliation of the invariant set $\Lambda$. One way to do that is to use once more the fact that we can take an arbitrarily large cover and use the hyperbolicity to argue that the intersections of the stable and unstable foliations of the invariant set $\Lambda$ with the gluing surfaces are $\eps$-dense. Thus we get non-empty intersections.
 This shows that there is only one chain recurrent class for the glued flow $\bar{\varphi}^t$, hence it is transitive.
\end{proof}

\begin{remark}
It would be quite interesting to pursue cut-and-glue constructions of (higher-dimensional) Anosov flows further. For instance, one might want to know whether passing to a finite cover in the above example is avoidable? Or whether the above example is, or can be made, volume preserving? Or whether one can glue multiple DA pieces together?
More generally, we would like to know if it is at all possible to make this type of construction starting from a different Anosov flow? In particular, deciding whether this construction can be somehow made to work starting from a suspension flow could help answer Question \ref{Q_generalized_Verjovsky}.
\end{remark}


\end{document}